\documentclass[11pt]{article}
\usepackage{amsmath,amssymb,amsthm,color,comment,geometry,graphicx,hyperref}
\usepackage[T1]{fontenc}
\usepackage[utf8]{inputenc}
\usepackage{authblk}
\usepackage{breqn}
\usepackage[title,titletoc,toc]{appendix}
\usepackage[utf8]{inputenc}

\newtheorem{theorem}{Theorem}[section]
\newtheorem{lemma}[theorem]{Lemma}
\newtheorem{proposition}[theorem]{Proposition}
\newtheorem{corollary}[theorem]{Corollary}
\theoremstyle{definition}
\newtheorem{definition}{Definition}[section]
\theoremstyle{remark}
\newtheorem{remark}{Remark}[section]

\newcommand{\Tr}{{\rm Tr}}
\newcommand{\tr}{{\rm tr}}
\newcommand{\ricden}{{\bf Ric}}
\newcommand{\ricfun}{\mathcal{R}ic}
\newcommand{\ric}{{\rm Ric}}
\newcommand{\scalar}{R}
\newcommand{\scalarfun}{\mathcal{R}}
\newcommand{\res}{{\rm Res}}
\newcommand{\pr}{{ Q}}
\newcommand{\A}{C(\mathbb{T}_\theta^3)}
\newcommand{\Ai}{C^\infty(\mathbb{T}_\theta^3)}
\newcommand{\nctorus}{{\mathbb{T}}^3_\theta}

\newcommand{\End}{{\rm End}}
\newcommand{\lap}{\triangle}

\title{The Ricci Curvature for  Noncommutative Three Tori}

\author{Rui Dong, Asghar Ghorbanpour, Masoud Khalkhali}
\affil{Department of Mathematics, University of Western Ontario}

\newcommand{\Addresses}{{% additional braces for segregating \footnotesize
  \bigskip
  \footnotesize

  Rui Dong, \textsc{Department of Mathematics, University of Western Ontario,
London, Ontario, Canada, N6A 5B7}\par\nopagebreak
  \textit{E-mail address}:\texttt{rdong7@uwo.ca}
  \medskip

  Asghar Ghorbanpour, \textsc{Department of Mathematics, University of Western Ontario,
London, Ontario, Canada, N6A 5B7}\par\nopagebreak
  \textit{E-mail address}:\texttt{aghorba@uwo.ca}

  \medskip

 Masoud Khalkhali, \textsc{Department of Mathematics, University of Western Ontario,
London, Ontario, Canada, N6A 5B7}\par\nopagebreak
  \textit{E-mail address}: \texttt{masoud@uwo.ca}

}}

\date{ }
\begin{document}
\maketitle

\allowdisplaybreaks

\begin{abstract}
     We compute the  Ricci  curvature of a  curved noncommutative  three torus. The computation is done  both for   conformal and  non-conformal perturbations of the flat metric. 
To perturb the flat metric,  the standard volume form on the noncommutative  three torus is 
 perturbed and the corresponding perturbed Laplacian is analyzed.
 Using Connes' pseudodifferential calculus for the noncommutative tori, we explicitly compute 
 the second   term of the  short  time heat kernel expansion for the perturbed Laplacians on functions and on 1-forms. The Ricci curvature is defined by localizing heat traces suitably. Equivalerntly, it  can be defined through  special values of localized spectral zeta functions. We also compute the scalar curvatures and compare our results with  previous calculations in the conformal case.  
Finally we compute the classical limit of our formulas and  show that  they  coincide with classical formulas in the commutative case.  
\end{abstract}

\tableofcontents{}

\section{Introduction}

The spectral geometry and study of local spectral invariants of curved noncommutative tori has been the subject of intensive studies in recent years.
In particular a Gauss-Bonnet theorem, the definition of scalar curvature,  and the computations of scalar curvature for noncommutative  two tori equipped with a curved metric has been achieved in \cite{Connes-Tretkoff2011,Fathizadeh-Khalkhali2013,Connes-Moscovici2014,Fathizadeh-Khalkhali2012}. 
Building on these results, 
computing the scalar curvature in other dimensions and settings is  carried  out in  \cite{Fathizadeh-Khalkhali2015,Lesch-Moscovici2016,Liu2015,Khalkhali-Motadelro-Sadeghi2016,Tanvir-Marcolli2012,Dbrowsli-Sitarz2015,Connes-Fathizadeh2016,Khalhali-Sitarz2018}.  
Beyond the scalar curvature, in \cite{Floricel-Ghorbanpour-Khalkhali2016} a definition of Ricci curvature in spectral terms is proposed and the Ricci density is computed for conformally flat  metrics on noncommutative two tori. 

In the present  work we shall compute the Ricci curvature  of noncommutative three tori for conformally flat  metrics as well as non-conformal   perturbations of the flat metric.  
Study of non-conformally flat  metrics  in three dimension is  justified since even in the commutative case  the class of conformally flat metrics on  a three dimensional manifold is much smaller than the class of all metrics.

At  the heart of our spectral formulation of the Ricci curvature  lies the  Weitzenb\"ock formula.
This formula measures how far the Laplacian on 1-forms is from the Bochner Laplacian of the Levi-Civita connection on the cotangent bundle.
It states \cite[Lemma 4.8.13]{Gilkey1995} that the difference is   the Ricci tensor lifted to an endomorphism of the cotangent bundle denoted by  $\ric$,  and called  the  Ricci operator in \cite{Floricel-Ghorbanpour-Khalkhali2016}.
More precisely, we have 
\begin{equation}\label{Weitzen}
\lap_1=\nabla^*\nabla +\ric.
\end{equation} 
This result combined with  Gilkey's formulas for the heat trace \cite{Gilkey1995} reveals immediately that a linear combination of the Ricci operator and the scalar curvature is the density of the second coefficient of the heat trace of the Laplacian on  1-forms. That is  
$$\Tr(e^{-t\lap_1})\sim a_0(\lap_1)t^{-m/2}+a_2(\lap_1)t^{1-m/2}+\cdots,\qquad t\to 0^+,$$
where 
$$a_2(\lap_1)=(4\pi)^{-m/2}\int_M \tr\big(\frac{1}{6}{\scalar}+\ric\big)\mathrm{dvol}_g,$$ 
and $\scalar$ denotes the scalar curvature.
These densities can be recovered by studying the localized heat trace $\Tr(Fe^{-t\lap_1})$, where $F$  is a  smooth endomorphisms of the cotangent bundle.
To isolate  the Ricci operator, the second density of the heat trace of the Laplacian on  functions $a_2(\lap_0)=(4\pi)^{-m/2}\frac{1}{6}{\scalar}$ enters the game where it is used to eliminate the scalar curvature present in $a_2(x,\lap_1)$.
Then   the Ricci functional,  as a functional on  the algebra of sections of the endomorphism bundle of the cotangent bundle of $M$,  is introduced as 
\begin{equation*}
\ricfun(F)=\lim_{t\to 0^+} t^{\frac{m}2-1} \left(\Tr(\tr(F)e^{-t\lap_0})-\Tr(Fe^{-t\lap_1})\right),\quad F\in C^\infty(\End(T^*M)).
\end{equation*}
If we denote the second density of the localized heat trace by  $a_2(\tr(F),\lap_0)$, the above formula can then be written as    
\begin{equation*}
\ricfun(F)=a_2(\tr(F),\lap_0)-a_2(F,\lap_1),\quad F\in C^\infty(\End(T^*M)).
\end{equation*}
An  equivalent version of the Ricci functional in terms of the spectral zeta function  can be  given by \cite{Floricel-Ghorbanpour-Khalkhali2016}
\begin{equation*}
\ricfun(F)=\begin{cases}
\zeta(0,\tr(F),\lap_0)-\zeta(0,F,\lap_1)+\Tr(\tr(F)\pr_0)-\Tr(FQ_1), & m=2\\ &\\ 
 \Gamma(\frac{m}{2}-1)\res_{s=\frac{m}{2}-1}\Big(\zeta(s,\tr(F),\lap_0)-\zeta(s,F,\lap_1)\Big), & m>2,
\end{cases}
\end{equation*}
where $\pr_j$ is the orthogonal projection on the kernel of $\lap_j$.

This paper is organized as follows. In Section 2, we  first recall the definition of the noncommutative Ricci curvature  from \cite{Floricel-Ghorbanpour-Khalkhali2016}.  To define the Ricci  functional  for the noncommutative three torus, it suffices to define the Laplacian on functions and on 1-forms. We also recall the rearrangement lemma and Connes' pseudodifferential calculus in this section. 
The analogue of the de Rham complex for the noncommutative three torus is discussed in Section \ref{derahm}.  For the analogues of  conformal $e^{-2h}(dx^2+dy^2+dz^2)$ and non-conformal  $e^{-2h}(dx^2+dy^2)+dz^2$ families of metrics,  the Laplacians  are computed in later sections. 
In Section \ref{conformalmetrics}, applying the pseudodifferential calculus,  the densities of the second terms are computed in the conformal case and  the scalar curvature and Ricci density are computed for these metrics in Proposition \ref{conformalalpha-1} and Theorems \ref{ncconformalscalarcurvature} and \ref{ncconformalRiccicurvature}. 
Finally in Section \ref{nonconformalmetrics} we first compute the scalar curvature of the  noncommutative three torus  equipped with a non-conformally flat metric.  We then compute the Ricci density for this class of metrics. It is interesting to note that  two of the functions that appear in the expression for scalar curvature, Theorem \ref{2nd_scalar_thm},  are  the same as functions that appear in the scalar curvature of the two 
dimensional curved noncommutative tori \cite{Connes-Moscovici2014,Fathizadeh-Khalkhali2013}. In Appendix A, we produce the steps that was used to compute the scalar curvature in the non-conformal case. In Appendix B,  we give the list of functions obtained from the rearrangement lemma that are used in our computations.

\section{Preliminaries}
In this section we shall fix notations and review preliminaries  required for the rest of the work. 
We will start with the definition of noncommutative three torus and then we construct the de Rham complex for it and discuss how one can define the Laplacians by fixing a metric  on the noncommutative torus. 
Finally, we recall the definition of the Ricci functional from \cite{Floricel-Ghorbanpour-Khalkhali2016} for noncommutative three tori. 
\subsection{Noncommutative three tori}
For a general introduction to topology and geometry of noncommutative tori the reader can consult  \cite{Connes1994}. Let $\theta = (\theta_{jk})\in M_{3}(\mathbb{R})$ be a skew-symmetric matrix. 
The noncommutative 3-torus $\A $ is the universal unital $C^\ast$-algebra generated by three unitary elements $u_1, u_2, u_3$ satisfying the relations:
\begin{equation*} 
u_{k}u_{j}=e^{2\pi i\theta_{jk}}u_{j}u_{k},\quad j, k=1, 2, 3.
\end{equation*}
We shall use both notations $\A$ and $\mathbb{T}^3_\theta$ to refer to the noncommutative space represented by the algebra $\A$.
For $\theta=0$, the $C^\ast$-algebra $\A$ is isomorphic to the algebra of  continuous functions on the 3-torus $\mathbb{T}^3=\mathbb{R}^3/\mathbb{Z}^3$.

There is an action of $\mathbb{T}^{3}$ on $\A$, which is given by the $3-$parameter group of automorphisms $\{\alpha_{z}\}$, such that
\begin{equation}\label{actionofTonA}
\alpha_{z}(u^{m})=z^{m}u^{m},
\end{equation}
where for $m=(m_{1}, m_{2}, m_{3})\in \mathbb{Z}^{3}$, we set $u^{m}=u_{1}^{m_{1}}u_{2}^{m_{2}}u_{3}^{m_{3}},$ and similarly, for $z=(z_{1}, z_{2}, z_{3})\in\mathbb{T}^{3}$, we denoted $z_{1}^{m_{1}}z_{2}^{m_{2}}z_{3}^{m_{3}}$ by $z^{m}$.
The set of all elements $a\in \A$ for which the map $z\mapsto \alpha_{z}(a)$ is smooth, form an involutive dense subalgebra of $\A$, which will be denoted by $\Ai$. 
Alternatively, $\Ai$ can be expressed as 
\begin{equation*}
\Ai = \Big\{\sum_{m\in \mathbb{Z}^{3}}a_{m}u^{m}:\{a_{m}\}_{m\in \mathbb{Z}^{3}} \text{ is rapidly decreasing}\Big\}.
\end{equation*}
By rapidly decreasing, we mean the Schwartz class condition that for all $k\in \mathbb{N}$,
\begin{equation*}
\sup_{m\in \mathbb{Z}^{3}}(1+|m|^{2})^{k}|a_{m}|^{2}<\infty.
\end{equation*}

There is a normalized faithful tracial state $\varphi$ on $\A$, determined by 
\begin{equation*}
\varphi(u^m)=0,  \quad \forall m\neq (0,0,0), \quad \text{and } \varphi(1)=1.
\end{equation*}
The tracial state $\varphi$ here plays the role of integration over $\nctorus$.
The algebra $\Ai$  possesses  three derivations, which are defined by the following relations:
\begin{equation*}
\delta_{j}(\sum_{m\in\mathbb{Z}^{3}}a_{m}u^{m})=\sum_{m\in\mathbb{Z}^{3}}m_{j}a_{m}u^{m}, \quad j = 1, 2, 3.
\end{equation*}
These derivations $\delta_j$ satisfy the relations
\begin{gather*}
(\delta_{j}(a))^{*}=-\delta_{j}(a^{*}),\\
\varphi(a\delta_{j}(b)) + \varphi(\delta_{j}(a)b)=0.
\end{gather*}

\subsection{De Rham complex for noncommutative three tori}\label{derahm}
We will first construct the space of differential forms on $\mathbb{T}^3_\theta$.
Let $W=\mathbb{C}^{3}$ and $\Lambda^\bullet W=\bigoplus_{j=0}^3 \Lambda^j W$ be the exterior algebra of $W$. 
The algebra 
\begin{equation*}
\Omega^\bullet\nctorus:=\Ai \otimes \Lambda^{\bullet}W,
\end{equation*}
will play the role of the algebra of complex differential forms of the noncommutative 3-torus.

We define the exterior derivative on functions, $d_{0}: \Omega^0\nctorus \to \Omega^1\nctorus $, by
\begin{equation*}
d_{0}(a)=(i\delta_{1}(a), i\delta_{2}(a), i\delta_{3}(a)), \quad \forall a\in\Ai.
\end{equation*}
Correspondingly, exterior derivative on 1-forms, $d_{1}: \Omega^1\nctorus \to \Omega^2\nctorus$,  and on 2-forms $d_{2}: \Omega^2\nctorus \to \Omega^3\nctorus$ are given by
\begin{equation*}
\begin{aligned}
d_1(a_{1}, a_{2}, a_{3})&=(i\delta_{1}(a_{2})-i\delta_{2}(a_{1}), i\delta_{2}(a_{3})-i\delta_{3}(a_{2}), i\delta_{1}(a_{3})-i\delta_{3}(a_{1})),\\
d_2(b_{1}, b_{2}, b_{3})&= i\delta_{1}(b_2)-  i\delta_{2}(b_3)+ i\delta_{3}(b_1).
\end{aligned}
\end{equation*}
It is not difficult to check that $d_{j+1}d_j=0$.
We define the de Rham complex of the noncommutative 3-torus to be the following complex
\begin{equation}\label{deRhamcomplex}
\Omega^0\nctorus\xrightarrow{d_0}\Omega^1\nctorus\xrightarrow{d_1} \Omega^2\nctorus\xrightarrow{d_2} \Omega^3\nctorus.
\end{equation}

In the commutative case, to define the Laplacian on forms, we need to fix a Riemannian metric first and  find the adjoint of the exterior derivatives, $d_j^\ast$ with respect to that metric. 
Then the Laplacian $\Delta_j$ on $j$-forms is defined as 
\begin{equation*}
\Delta_j=d_{j-1}d_{j-1}^*+d_j^*d_j.
\end{equation*} 
In the noncommutative case  we can study specific forms of metrics where the effect of the metric can be implemented through a volume form.
Then this helps us to define the adjoint of the exterior derivatives and similar to the classical case, one can define the Laplacian on $j$-forms.
These metrics include conformal perturbation of a flat metric, as it is studied in \cite{Connes-Tretkoff2011,Fathizadeh-Khalkhali2012,Connes-Moscovici2014,Fathizadeh-Khalkhali2013} for noncommutative two tori, and a new class of  non-conformally flat metrics in which only two directions are perturbed by a conformal factor.  
The geometry of conformally flat metrics on $\mathbb{T}^3_{\theta}$ will be studied in section \ref{conformalmetrics},  and  the geometry of non-conformally flat metrics  will be studied in section \ref{nonconformalmetrics}.

\subsection{The Ricci functional}\label{Riccifunctionalsection}
In a noncommutative setting, as a general rule, spectral methods must be employed to formulate metric invariants.
For example, in the noncommutative formulation of the Ricci curvature in \cite{Floricel-Ghorbanpour-Khalkhali2016}, instead of a tensorial algebraic definition, the spectral properties of the Laplacians are used to define and compute what is called  the Ricci density.
This formulation allows us to define this quantity for the noncommutative three torus. 
In this section, we quickly review the definitions and motivations for this new formulation.   

Suppose $M$ is an $m-$dimensional closed oriented Riemannian manifold. 
Let $V\rightarrow M$  be  a smooth Hermitian vector bundle over $M$ and $P: C^{\infty}(V)\rightarrow C^{\infty}(V)$ be a positive elliptic differential operator of order $d$. 
The heat operator $e^{-tP}$ is trace class for all positive values of $t$ and it has a short time  asymptotic expansion (cf. \cite{Gilkey1995})
\begin{equation*}
\Tr(e^{-tP})\sim \sum_{n=0}^{\infty}a_n(P)t^{\frac{n-m}{d}}, \qquad t\to 0^+.
\end{equation*} 
The coefficients  $a_n(P)$ are given by an integral formula 
\begin{equation}\label{integralformulaforheattracecoe}
a_n(P) = \int_M\tr(a_n(x, P))\mathrm{dvol}(x),
\end{equation}
where $\tr(a_n(x, P))$ is the fibrewise trace and $\mathrm{dvol}(x)=\sqrt{\mathrm{det}g}\, dx^1...dx^m$ is the Riemannian volume form of $M$.

To recover the densities $a_n(x,P)$, one needs to study the localized heat trace $\Tr(Fe^{-tP})$ by a localizing factor $F\in C^{\infty}(\mathrm{End}(V))$.
For an endomorphism $F\in C^{\infty}(\mathrm{End}(V))$, there is also a complete asymptotic expansion 
\begin{equation*}
\Tr(Fe^{-tP})\sim \sum_{n=0}^{\infty}a_n(F,P)t^{\frac{n-m}{d}}, \qquad t\to 0^{+},
\end{equation*}
where, this time the coefficients $a_n(F,P)$ can be written as the integral
\begin{equation*}
a_n(F,P)=\int_{M}\mathrm{tr}(F(x)a_n(x,P))\mathrm{dvol}(x).
\end{equation*}
A  method to compute these densities,  which uses the pseudodifferential calculus,  will be outlined in the next section, and will be used for differential operators on the noncommutative tori.
 
On the other hand, if $P$ is a Laplace type operator, namely, a positive elliptic operator whose leading symbol is given by the inverse of the metric tensor, then there exists a unique connection $\nabla$ on $V$ and a unique endomorphism $E\in C^{\infty}(\mathrm{End}(V))$ such that \cite{Gilkey1995}
\begin{equation*} 
P = P_{\nabla} - E,
\end{equation*}
where $P_{\nabla}: C^{\infty}(V)\rightarrow C^{\infty}(V)$ is the Bochner Laplacian of the connection defined as the composition of operators as follows 
\begin{equation*}
P_{\nabla}: C^{\infty}(V)\xrightarrow{\nabla}C^{\infty}(T^*M\otimes V)\xrightarrow{\nabla}C^{\infty}(T^*M\otimes T^*M\otimes V)\xrightarrow{-g\otimes 1}C^{\infty}(V).
\end{equation*} 
The first two densities of the corresponding heat kernel for $P$ are given by 
\begin{equation*}
\begin{aligned}
a_0(x, P)&=(4\pi)^{-m/2}\rm{I},\\
a_2(x, P)&=(4\pi)^{-m/2}(\frac{1}{6}R(x)+E),
\end{aligned}
\end{equation*}
where  $R(x)$ is the scalar curvature of $M$.

We apply the above general idea to Laplacians
$\lap_0$ and $\lap_1$ on $\Omega^0(M)$ and $\Omega^1(M)$.
The endomorphism $E$ for the Laplacian on functions $\Delta_0$ is zero, therefore, the first two densities in the heat kernel of $\lap_0$ are given by
\begin{equation*}
\begin{aligned}
a_0(x, \lap_0)&=(4\pi)^{-m/2},\\
a_2(x, \lap_0)&=(4\pi)^{-m/2}(\frac{1}{6}R(x)).
\end{aligned}
\end{equation*}
By Weitzenb\"{o}ck formula \eqref{Weitzen}, the endomorphism for Laplacian on 1-forms $\lap_1$ is $-\mathrm{Ric}_{x}$, the Ricci operator on the cotangent bundle. 
Thus we have 
\begin{equation*}
\begin{aligned}
&a_0(x, \lap_1)=(4\pi)^{-m/2}\rm{I},\\
&a_2(x, \lap_1)=(4\pi)^{-m/2}(\frac{1}{6}R(x)-\mathrm{Ric}_{x}).
\end{aligned}
\end{equation*}
These observations lead us to the following  definition from \cite{Floricel-Ghorbanpour-Khalkhali2016}.
\begin{definition}\label{Riccifunctionaldef}
The Ricci functional $\ricfun:C^{\infty}(\mathrm{End}(T^*M))\to \mathbb{C}$
is defined as 
\begin{equation}\label{Riccifunctional}
\ricfun(F)=a_2(\mathrm{tr}(F), \lap_0)-a_2(F,\lap_1).
\end{equation}
\end{definition}
The Ricci functional can also be described in terms of the spectral zeta function \cite[Proposition 2.2]{Floricel-Ghorbanpour-Khalkhali2016}: 
\begin{equation}\label{zetaformulationofRicci}
\ricfun(F)=\begin{cases}
\zeta(0,\tr(F),\lap_0)-\zeta(0,F,\lap_1)+\Tr(\tr(F)\pr_0)-\Tr(FQ_1), & m=2\\  
\Gamma(\frac{m}{2}-1){\res_{s=\frac{m}{2}-1}}\left( \zeta(s,\tr(F),\lap_0)- \zeta(s,F,\lap_1)\right), & m>2.
\end{cases}
\end{equation}
Here $\zeta(s,F,\lap_1)$ is the localized spectral zeta function defined by $\Tr(F\lap_1^{-s})$ for $\Re (s)> m/2$, $\zeta(s,f,\lap_0)$ is defined similarly, and $\pr_j$ is the orthogonal projection on the kernel of $\lap_j$.

\subsection{Pseudodifferential calculus and local computations}\label{Pseudoandspectral}
In this section, we briefly recall the definition of Connes’ pseudodifferential calculus \cite{Connes1980} for $C^{\ast}-$dynamical systems adapted to 3-dimensional noncommutative tori and outline the necessary steps to use it to compute the heat trace densities. 
These densities then can be used to define  the Ricci density and the scalar curvature density for the noncommutative three torus. 

The action \eqref{actionofTonA} on $\A$ defines  a $C^{\ast}-$dynamical system $(\A, \mathbb{R}^3, \alpha)$. 
A pseudodifferential calculus can be assigned to the given $C^\ast$-dynamical system.
The symbols of order $d$ are given by smooth maps $\rho: \mathbb{R}^3 \rightarrow \Ai$ such that
\begin{enumerate}
\item For any non-negative multi-indices $\alpha, \beta$, there exists a positive number $C_{\alpha, \beta}$ such that 
\begin{equation*}
\|\delta^\alpha \partial^{\beta}\rho(\xi)\|\leqslant C_{\alpha,\beta}(1+|\xi|)^{d-|\beta|}.
\end{equation*}
\item There is a smooth map $f: \mathbb{R}^{3}\backslash \{0\}\rightarrow \Ai$ such that
\begin{equation*}
\lim_{\lambda\rightarrow \infty}\lambda^{-d}\rho(\lambda\xi_1, \lambda\xi_2, \lambda\xi_3)=f(\xi_1, \xi_2, \xi_3).
\end{equation*}
\end{enumerate}
Here, we use the notation that for any multi-index $\alpha=(\alpha_1, \alpha_2, \alpha_3)$ we have
\begin{equation*}
\partial^\alpha=\frac{\partial^{\alpha_1}}{\partial \xi_{1}^{\alpha_1}}\frac{\partial^{\alpha_2}}{\partial \xi_{2}^{\alpha_2}}\frac{\partial^{\alpha_3}}{\partial \xi_{3}^{\alpha_3}},
\quad 
\delta^\alpha = \delta_1^{\alpha_1}\delta_2^{\alpha_2}\delta_3^{\alpha_3}.
\end{equation*}

We shall denote the set of all symbols of order $d$ by $S^{d}(\nctorus)$.
The pseudodifferential operator associated to a given symbol $\rho\in S^d(\nctorus)$ is defined by
\begin{equation*}
P_\rho(a) = (2\pi)^{-3}\int\int e^{-iz\cdot\xi}\rho(\xi)\alpha_z(a)dzd\xi, \quad a\in \Ai.
\end{equation*}
The following theorem from \cite{Connes1980} gives a formula for the symbol of the product of pseudodifferential operators.
\begin{theorem}\label{product}
If $\rho_j\in S^{d_j}(\nctorus)$, $j=1,2,$  there exists a $\rho\in S^{d_1 + d_2}$ such that $P_\rho = P_{\rho_1}P_{\rho_2}$, 
and moreover, $\rho$ has an asymptotic expansion given by
\begin{equation}\label{productsymbol}
\rho\sim \sum_{\alpha}\frac{1}{\alpha!}\partial^{\alpha}(\rho_1)\delta^{\alpha}(\rho_2).
\end{equation}
\end{theorem}
\begin{remark}
 For our purposes, we need more general symbols which take values in $\Ai\otimes M_n(\mathbb{C})$.
The above calculus easily extends to this setting.
\end{remark}
In the rest of this section we outline the steps through which one can find the second density of the heat trace  $a_2$ for a positive elliptic differential operator on $\nctorus$ using the pseudodifferential calculus.
For more details we refer the readers to \cite{Gilkey1995} for the commutative case and \cite{Connes-Tretkoff2011,Fathizadeh-Khalkhali2013,Connes-Moscovici2014} for the noncommutative case.

Let $P$ be a second order positive elliptic operator on $\mathbb{T}_\theta^3$ with positive principal symbol, i.e. if we write the symbol of $P$ as the sum of the homogeneous parts $a_2(\xi)+a_1(\xi)+a_0(\xi)$, $a_2(\xi)$ is positive and it is invertible for any nonzero $\xi\in\mathbb{R}^3$.
Then  the parametrix $(P-\lambda)^{-1}$ for any $\lambda\in \mathbb{C}\backslash \mathbb{R}^{+}$ is a pseudodifferential operator of order $-2$ and its symbol $\sigma((P-\lambda)^{-1})$ can be written as $b_0(\xi,\lambda)+b_1(\xi,\lambda)+\cdots$, where $b_j(\xi,\lambda)$ is homogeneous of order $-2-j$ in $(\xi, \lambda)$, that is it satisfies $b_j(t\xi, t^2\lambda)=t^{-2-j}b_j(\xi, \lambda)$ for all $t\geq 0$.
The terms $b_j$  can be written in terms of $a_j$'s and $b_0$ using the recursive formula for symbol product \eqref{productsymbol} applied to the equality $(P-\lambda)^{-1}(P-\lambda)\sim 1$:
\begin{equation}\label{bjs}
\begin{aligned}
b_0(\xi, \lambda)=&(a_2-\lambda)^{-1},\\
b_1(\xi, \lambda)=&- b_0a_1 b_0-\sum_{j=1}^3\partial_{j}(b_0)\delta_j(a_2)b_0,\\
b_2(\xi, \lambda)
=&-b_0 a_0b_0-b_1 a_1b_0\\
&-\sum_{i=1}^3 \Big(\partial_i(b_0)\delta_i(a_1)b_0+ \partial_i(b_1)\delta_i(a_2)b_0+\frac12 \partial_i\partial_j(b_0)\delta_i\delta_j(a_2)b_0\Big).
\end{aligned}
\end{equation}

Using the Cauchy integral formula and the formula for the trace in terms of the symbols of a smoothing operator, one has the asymptotic expansion of the localized heat trace $\Tr(Fe^{-tP})$ as follows:
\begin{equation*}
\Tr(Fe^{-tP})\sim \sum_{n=0}^\infty t^{\frac{n-3}{2}} \varphi\Big(\tr\big(F \frac{1}{(2\pi)^3}\int_{\mathbb{R}^3}\frac{1}{2\pi i}\int_\gamma e^{-\lambda}  b_n(\xi,\lambda) d\lambda d\xi\big)\Big).
\end{equation*}
The geometric meaning of the second density $a_2(P)$, i.e. densities for the coefficient of the term $t^{-\frac12}$, in the classical case is discussed in section \ref{Riccifunctionalsection}.
In the noncommutative case, by analogy, the second density which is given by    
\begin{equation}\label{secondtermintegralform1}
a_2(P)=\frac{1}{(2\pi)^3}\int_{\mathbb{R}^3}\frac{1}{2\pi i}\int_\gamma e^{-\lambda}  b_2(\xi,\lambda) d\lambda d\xi,
\end{equation} 
can be used to define the Ricci and scalar curvature for the noncommutative torus when $P$ is a carefully chosen geometric operator.
By a homogeneity argument given in \cite{Khalkhali-Motadelro-Sadeghi2016} for noncommutative three tori, we can rewrite $a_2(P)$ as 
\begin{equation}\label{secondtermintegralform}
a_2(P)=\frac{1}{8\pi^{7/2}}\int_{\mathbb{R}^3} b_2(\xi,-1)d\xi.
\end{equation}

To compute the integral \eqref{secondtermintegralform} above, one needs to apply the rearrangement lemma.
Here we shall use a general version from \cite[Corollary 3.5]{Lesch2017}. 
\begin{proposition}\label{leschrearrangemnet}
Suppose $\mathcal{A}$ is a $C^{\ast}-$algebra.
Let $f_0,...,f_p: \mathbb{R}_{\geqslant 0}\rightarrow \mathbb{C}$ be smooth functions such that for each pair of positive numbers $0<C_1<C_2$ and each multi-index $\alpha\in \mathbb{N}^{n+1}$, the function $f(x_0,...,x_p):=\prod_{j=0}^{p}f_j(x_j)$ satisfies
\begin{equation*}
\int_{0}^{\infty}\sup_{\substack{C_1\leqslant s_j\leqslant C_2\\ 0\leqslant j \leqslant n}} |u^{|\alpha|}(\partial^{\alpha}f)(us)|du < \infty,
\end{equation*}
Let $A=e^a$ for some selfadjoint element $a\in \mathcal{A}$. 
Then for $\rho_1,\cdots,\rho_p\in \mathcal{A}$
\begin{equation*}
\begin{aligned}
&\int_{0}^{\infty}f_0(uA)\cdot b_1 \cdot f_1(uA)\cdot \cdots \cdot b_p\cdot f_p(uA)du\\
&=A^{-1}F_{\gamma}(\Delta_{(1)}, \Delta_{(1)}\Delta_{(2)},\cdots, \Delta_{(1)}\cdots \Delta_{(p)})(\rho_1 \cdot\rho_2 \cdots \cdot \rho_p),
\end{aligned}
\end{equation*}
where $\Delta_{(j)}$ is the modular operator acting on $b_j$ by $\Delta(b)=A^{-1}bA$, and the smooth function $F$ is given by
\begin{equation*}
F(s_1,..., s_p) = \int_{0}^{\infty} f_0(u)\cdot f_1(us_1)\cdot \cdots\cdot f_p(us_p)du. 
\end{equation*}
\end{proposition}

In the following,  we first compute the Laplacians  $\lap_{0,h}$ and $\lap_{1,h}$ and show that they are anti-unitary equivalent to  operators $\tilde{\lap}_{0,h}$ and $\tilde{\lap}_{1,h}$ which are  second order positive elliptic differential operators.
Hence, the above theory can be applied  to find their second densities $a_2(\tilde{\lap}_{0,h})$ and $a_2(\tilde{\lap}_{1,h})$. 
Now we can define   
\begin{definition}\label{ncscalar}
The scalar curvature functional $\scalarfun:\Ai\to \mathbb{C}$ is defined as
\begin{equation}
\scalarfun(a):=\varphi(aa_2(\lap_{0,h})),\qquad a\in\Ai,
\end{equation}
and $a_2(\lap_{0,h})$ will be called  the scalar curvature density or just the scalar curvature and we denote it by $\scalar$.
\end{definition}
\noindent Similar to  Definition \ref{Riccifunctionaldef}, we define
\begin{definition}\label{ncricci}
The Ricci curvature functional $\ricfun:\Ai\otimes M_3(\mathbb{C})\to \mathbb{C}$ is defined as
\begin{equation}
\ricfun(F):=\varphi(\tr(F)a_2(\lap_{0,h}))-\varphi(Fa_2(\lap_{1,h})),\qquad F\in\Ai\otimes M_3(\mathbb{C}).
\end{equation}
The Ricci density is then defined by the equation
$$\ricfun(F)=\varphi(\tr(F\ricden)),\qquad F\in\Ai\otimes M_3(\mathbb{C}).$$  
It can be readily seen that 
$$\ricden=\scalar\otimes \mathrm{I}_3-a_2({\lap}_{1,h}).$$
\end{definition}
\noindent Using the Mellin transform, one can show that the above definition is equivalent to the equation \eqref{zetaformulationofRicci}.
\begin{remark}\label{thediffofcoeff}
Note that we choose to drop the effect of the volume form density $vol$ on the Ricci and scalar curvature densities. We have also dropped the overall  multiplicative constants in our definitions above.
This means that we are ignoring a factor of $\frac{1}{48\pi^{3/2}}vol$ for the scalar  curvature density and a factor of $\frac{1}{8\pi^{3/2}}vol$ for the Ricci density.
Moreover, we shall use operators which are anti-unitary equivalent to the Laplacians while computing the densities.
It can be seen readily that if $\tilde{\lap}=U^*\lap U$, for some anti-unitary operator $U$ then 
$$\Tr(Fe^{-t\tilde{\lap}})=\Tr(FU^*e^{-t\lap}U)=\Tr(UFU^*e^{-t\lap}).$$  
Similarly, the localized heat trace densities are related as above.
These two points should be taken into account while we recover  the classical results in the limit $\theta\to 0$ of our formulas for the noncommutative tori.
\end{remark}

\section{Ricci density for conformally flat metrics}\label{conformalmetrics}
In this section we first investigate how the geometry of  conformally flat metrics on three torus $\mathbb{T}^3$ can be implemented on the noncommutative three tori $\mathbb{T}_\theta^3$.
We then use it to define the Laplacian on functions and on 1-forms; that is we find the Laplacian of the de Rham complex \eqref{deRhamcomplex} with respect to the induced inner products. 
Then using the pseudodifferential calculus we compute the second densities of heat trace asymptotic for these operators which by Definitions \ref{ncscalar} and \ref{ncricci} can be used to define  the scalar curvature density and the Ricci curvature density for $\nctorus$.

In the commutative case, if $h\in C^{\infty}(M)$ is a real valued function,  conformally changing the Riemannian metric by the function $e^{-2h}$ will result in changing the volume form. 
For instance, if the dimension of a closed Riemannian manifold $M$ is $m$, and we denote the conformal change of $g$ by $\tilde{g}=e^{-2h}g$, then the new volume form $\tilde{dx}$ is $e^{-mh}dx$.
As a result, the inner products on $\Omega^{0}(M)$, $\Omega^{1}(M)$, and $\Omega^{2}(M)$ are given by
\begin{equation*}
\begin{aligned}
\langle f_{1}, f_{2} \rangle_{\tilde{g}} &= \int_{M}f_{1}\bar{f_{2}}e^{-mh}dx,\\ 
\langle \alpha_{1}, \alpha_{2}\rangle_{\tilde{g}} &= \int_{M}g^{-1}(\alpha_{1}, \bar{\alpha}_{2})e^{(2-m)h}dx, \\
\langle \omega_{1}, \omega_{2}\rangle_{\tilde{g}} &= \int_{M}(\wedge^{2}g^{-1})(\omega_{1}, \bar{\omega}_{2})e^{(4-m)h}dx. 
\end{aligned}
\end{equation*}

Inspired by these classical equations, we are able to study the conformal change of metrics for  noncommutative three tori.  
Let $h$ be a self-adjoint positive element of $\Ai$ 
and let  $\varphi_{0}(a)=\varphi(ae^{-3h}),$ for any $a\in \A$.
Denote the Hilbert space given by the GNS construction of $\A$ with respect to the positive linear functional $\varphi_{0}$ by $\mathcal{H}_{h}^{(0)}$.
In other words, the inner product of $\mathcal{H}_{h}^{(0)}$ is given by
\begin{equation*}
\langle a, b\rangle_{0,h} = \varphi(b^{*}ae^{-3h}).
\end{equation*}
Let $\mathcal{H}_{h}^{(1)}$ denote the Hilbert space completion of $\Omega^1\mathbb{T}_\theta^3$ with respect to the inner product of  $\mathcal{H}_{h}^{(1)}$ given by
\begin{equation*}
\langle (a_{1}, a_{2}, a_{3}), (b_{1}, b_{2}, b_{3}) \rangle_{1,h} = \varphi(\sum_{i=1}^{3}b_{i}^{*}a_{i}e^{-h}).
\end{equation*}
Similarly, let $\mathcal{H}_{h}^{(2)}$ denote the Hilbert space completion of  $\Omega^2\mathbb{T}_\theta^3$ with respect to the inner product of $\mathcal{H}_{h}^{(2)}$ given by
\begin{equation*}
\langle (a_{1}, a_{2}, a_{3}), (b_{1}, b_{2}, b_{3}) \rangle_{2,h} = \varphi(\sum_{i=1}^{3}b_{i}^{*}a_{i}e^{h}).
\end{equation*}

We identify the formal adjoint operator $d_j^{*}$ of $d_j$ acting on elements of $\Omega^{j+1}\nctorus\subset \mathcal{H}^{(j+1)}_h$ as follows. 
Let us denote $e^{h/2}$ by $k$. Then  we have
$$d_{0}^{*}(a_1,a_2,a_3) = -i\sum_{j=1}^{3} \delta_{j} (a_{j}k^{-2})k^{6},$$
and
\begin{align*}
&d_{1}^{*}(a_1,a_2,a_3)=\\
&\Big(i\delta_{3}(a_{3}k^{2})k^{2} + i\delta_{2}(a_{1}k^{2})k^{2}, i\delta_{3}(a_{2}k^{2})k^{2} - i\delta_{1}(a_{1}k^{2})k^2,-i\delta_{1}(a_{3}k^{2})k^{2} - i\delta_{2}(a_{2}k^{2})k^{2}\Big).
\end{align*}
Now, we can define the Laplacian on 0-forms to be $\lap_{0, h} = d_{0}^{*}d_{0}$, and the Laplacian on 1-forms to be  $\Delta_{1, h}=d_{1}^{*}d_{1} + d_{0}d_{0}^{*} $. 
We have 
\begin{equation*}
\lap_{0, h}(a)
= d_{0}^{*}d_{0}(a)
=\sum_{j=1}^3\delta_{j}(\delta_{j}(a)k^{-2})k^{6},
\end{equation*}
On the other hand, 
the Laplacian on 1-forms is given by
\begin{eqnarray*}
&\lap_{1, h}\left(a_{1}, a_{2}, a_{3}\right)=&\\
\Big(\hspace*{-0.2cm}&
\left(\delta_2(\delta_2(a_1)k^2)+\delta_3(\delta_3(a_1)k^2)-\delta_2(\delta_1(a_2)k^2)-\delta_3(\delta_1(a_3)k^2)\right)k^2+{\textstyle{\sum} \,}\delta_1(\delta_j(a_jk^{-2})k^6),&\\
&\left(\delta_1(\delta_1(a_2)k^2)-\delta_1(\delta_2(a_1)k^2)+\delta_3(\delta_3(a_2)k^2)-\delta_3(\delta_{2}(a_{3})k^{2})\right)k^{2}+{\textstyle{\sum} \,}\delta_2(\delta_j(a_jk^{-2})k^6),&\\
&\left(\delta_1(\delta_1(a_3)k^2)-\delta_1(\delta_3(a_1)k^2)-\delta_2(\delta_3(a_2)k^2) + \delta_2(\delta_2(a_3)k^2)\right)k^2 + {\textstyle{\sum} \,}\delta_3(\delta_j(a_jk^{-2})k^6)&\hspace*{-0.2cm}\Big).
\end{eqnarray*}

The right multiplication operator $R_{k^{3}}$ satisfies the property 
\begin{equation*}
\begin{aligned}
\langle R_{k^{3}}a, R_{k^{3}}b\rangle_{0, h} = \varphi_{0}(k^{3}b^{*}ak^{3})= \varphi(k^{3}b^{*}ak^{-3})= \varphi(b^{*}a)= \langle a, b\rangle_{0, 0},
\end{aligned}
\end{equation*}
and thus extends  to a unitary operator from $\mathcal{H}_{0}^{(0)}$ to $\mathcal{H}_{h}^{(0)}$, which we still denote  by $R_{k^{3}}$.
Let $J: \A\rightarrow \A$ be the adjoint map $J(a)=a^{*}$. Then 
%\begin{equation*}
$R_{k^{3}}J: \mathcal{H}_{0}^{(0)}\rightarrow \mathcal{H}_{h}^{(0)}$
%\end{equation*} 
is an anti-unitary operator.  Thus $\lap_{0, h}$ is anti-unitary equivalent to
\begin{equation*}
\tilde{\lap}_{0,h}:=JR_{k^{3}}^{*}\lap_{0, h}R_{k^{3}}J%&= JR_{k^{3}}^{*}JJ\lap_{0, h}JJR_{k^{3}}J\\
= k^{-3}(J\lap_{0, h}J)k^{3}=\sum_{j=1}^{3}k^{3}  \delta_{j}  k^{-2}  \delta_{j}  k^{3}.
\end{equation*}

It can also be seen that 
\begin{equation*}
\quad \langle R_{k}(a_1,a_2,a_3), R_{k}(b_1,b_2,b_3) \rangle_{1, h}
= \langle (a_1,a_2,a_3),(b_1,b_2,b_3) \rangle_{1, 0}.
\end{equation*}
Hence $R_{k}$ can be extended to a unitary operator from $\mathcal{H}_{0}^{(1)}$ to $\mathcal{H}_{h}^{(1)}$, which we still denote  by $R_{k}$.
Then we get an anti-unitary operator $ R_{k}J: \mathcal{H}_{0}^{(1)}\rightarrow \mathcal{H}_{h}^{(1)}.$
Therefore, $\lap_{1, h}$ is anti-unitary equivalent to
\begin{equation*}
\tilde{\lap}_{1,h}:=JR_{k}^{*}\lap_{1, h}R_{k}J
=k^{-1}J\lap_{1,h}Jk.
\end{equation*}
Since $JR_{k^{m}}J=k^{m}$, and  $J\delta_{j} = -\delta_{j}J$, for $j=1, 2, 3,$  we have 
\begin{equation*}
\quad JR_{k^{m}}\delta_{i}R_{k^{n}}\delta_{j}J
=JR_{k^{m}}JJ\delta_{i}R_{k^{n}}\delta_{j}J
=k^{m}\delta_{i}k^{n}\delta_{j}.
\end{equation*}
Thus, 
\begin{eqnarray*}
&\tilde{\lap}_{1,h}(a_1, a_2, a_3)=&\\
\Big(&
k\delta_3 k^2 \delta_3 k a_1 + k \delta_2 k^2 \delta_2 k a_1- k \delta_2 k^2 \delta_1 k a_2 - k\delta_3 k^2 \delta_1 k a_3 + \sum k^{-1}\delta_1 k^6 \delta_j k^{-1} a_j,&\\
&- k \delta_1 k^2\delta_2ka_1 +k \delta_3 k^2 \delta_3 k a_2 + k \delta_1 k^2\delta_1k a_2 -k \delta_3 k^2 \delta_2k a_3 + \sum k^{-1}\delta_2k^6 \delta_j k^{-1}a_j,&\\
&-k\delta_1 k^2 \delta_3 k a_1+k\delta_1 k^2\delta_1k a_3 - k\delta_2 k^2 \delta_3k a_2+ k\delta_2 k^2 \delta_2 k a_3 + \sum k^{-1} \delta_3 k^6 \delta_j k^{-1} a_j&\Big).
\end{eqnarray*}

\subsection{Scalar curvature}
The scalar curvature for conformally flat  metrics on noncommutative three tori was first computed in \cite{Khalkhali-Motadelro-Sadeghi2016}. 
For the sake of  completeness, we shall compute it again here.
As discussed in Section 2.4, we define the scalar curvature of $\nctorus$ to be 
\begin{equation}\label{scalarcurvature2}
\scalar=a_2(\tilde{\lap}_{0,h})=\frac{1}{8\pi^{7/2}} \int_{\mathbb{R}^3}b_2(\xi, -1) d\xi.
\end{equation}
where $b_2(\xi,-1)$ is the second term in the asymptotic expansion of the symbol of the parametrix of $\tilde{\lap}_{0,h}$.

To compute $b_2$ we need first to find the symbol of the Laplacian on functions.
\begin{lemma}\label{symboloflap0noncon}
Let the symbol of $\tilde{\lap}_{0,h}$ be written as the sum of its homogeneous parts,
$\sigma(\tilde{\lap}_{0,h}) = a_2 + a_{1} + a_0$.
Then we have
\begin{equation*}
\begin{aligned}
&a_2=k^4 \xi _1^2+k^4 \xi _2^2+k^4 \xi _3^2,\\
&a_1=\sum_{i=1}^{3}(2 k\delta _i(k^3)+k^3\delta _i(k^{-2})k^3)\xi_i,\\
&a_0=\sum_{i=1}^{3}\Big(k\delta _i^2(k^3) + k^3\delta _i(k^{-2})\delta _i(k^3)\Big).
\end{aligned}
\end{equation*}\qed
\end{lemma}

To evaluate the integral in \eqref{scalarcurvature2}, for this case, we shall first move to spherical coordinates.
After performing the angular integrals, we are left with sums of integrals of the form
$$\int_{0}^{\infty}b_0^{m_0}\rho_1 b_0^{m_1}\rho_2b_0^{m_2}\cdots \rho_lb_0^{m_p} u^{(-3/2 + \sum_{j=0}^{p}m_j)}du.$$
To compute these latter integrals we need to use the following version of the rearrangement lemma.
Here we present it as a corrollary of Proposition \ref{leschrearrangemnet}, but a straightforward  proof can be found in \cite{Khalkhali-Motadelro-Sadeghi2016}.
\begin{corollary}\label{rearrangecon}
Let $b_0=(1+k^4 u)^{-1}$, $\rho_j\in \Ai$, $m_j\in \mathbb{Z}$, for $j=0, 1,..., p,$ and  set the modular operator $\Delta$ be $\Delta(x)=k^{-6}xk^6.$ Then 
\begin{equation*}
\int_{0}^{\infty}b_0^{m_0}\rho_1 b_0^{m_1}\cdots \rho_lb_0^{m_p} u^{(-\frac32+\sum m_j)}du=k^{(2-4\sum m_j)}F_{m_0,,...,m_p}(\Delta_{(1)},\cdots,\Delta_{(p)})(\rho_1\cdot\rho_2 \cdots \rho_p),
\end{equation*}
where 
\begin{equation*}
F_{m_0,\cdots ,m_p}(s_1,\cdots,s_p)=
\int_{0}^{\infty}(1+u)^{-m_0}\prod_{j=1}^{p}\Big( u\prod_{h=1}^{j}s_{h}^{\frac23} + 1\Big)^{-m_j}u^{(-\frac32+\sum m_j)}du.
\end{equation*}
\end{corollary}
\begin{proof}
Let $u$ be $t^{2/3}$. 
Then we have  $b_0=(1+(tA)^{2/3})^{-1}$ where $A=k^6=e^{3h}$, and it is enough to consider the following functions;  
\begin{equation*}
\begin{aligned}
f_0(x):=x^{-4/3+3/2\sum_{j=0}^{p}m_j}(1+x^{2/3})^{-m_0},\qquad f_{j}(x):=(1+x^{2/3})^{-m_j}, \quad j=1,...,p.
\end{aligned}
\end{equation*}
If we set  $F_{m_0,\cdots,m_p}(s_1,..., s_p)=F (s_1,s_1s_2,\cdots, s_1\cdots s_p)$, by Proposition \ref{leschrearrangemnet}, the result is proven. 
\end{proof}
For instance
\begin{align*}
F_{1,1}(s_1)&=\frac{\pi }{s_1^{\frac23}+\sqrt[3]{s_1}},\qquad 
F_{2,1}(s_1)=\frac{\pi  \left(\sqrt[3]{s_1}+2\right)}{2 \left(\sqrt[3]{s_1}+1\right)^2 \sqrt[3]{s_1}},\\
F_{1,1,1}(s_1,s_2)&=\frac{\pi  \left(\sqrt[3]{s_1} \left(\sqrt[3]{s_2}+1\right)+1\right)}{\left(\sqrt[3]{s_1}+1\right) s_1 \left(\sqrt[3]{s_2}+1\right) \sqrt[3]{s_2} \left(\sqrt[3]{s_1} \sqrt[3]{s_2}+1\right)},\\
F_{2,1,1}(s_1,s_2)&=\frac{\pi  \left(\left(\sqrt[3]{s_1}+2\right) \sqrt[3]{s_1} \left(\sqrt[3]{s_2}+1\right) \left(\sqrt[3]{s_1} \sqrt[3]{s_2}+2\right)+2\right)}{2 \left(\sqrt[3]{s_1}+1\right)^2 s_1 \left(\sqrt[3]{s_2}+1\right) \sqrt[3]{s_2} \left(\sqrt[3]{s_1} \sqrt[3]{s_2}+1\right)^2}.
\end{align*}
The complete list of these functions can be found in Appendix \ref{appenfunctionsfromrearrange}.

All the $\rho_j$'s appeared in our computations are multiples of $\delta_j(k)$ or $\delta_j^2(k)$. 
We want to write all $\rho_j$'s in terms of $\log k$. 
To perform this step, using the expansional formula applied in  \cite[section 6.1]{Connes-Moscovici2014}, we find the corresponding formula  
\begin{equation*}
\begin{aligned}
&k^{-1}\delta_j(k)=f(\Delta)(\delta_j(\log k)),\\
&k^{-1}\delta_j^2(k)=f(\Delta)(\delta_j^2(\log k))+2 g(\Delta_{(1)},\Delta_{(2)})(\delta_j(\log k) \cdot \delta_j(\log k)),
\end{aligned}
\end{equation*}
where,
$$f(x)=\int_0^1 x^{s/6}ds=\frac{6(x^{1/6}-1)}{\log x },$$
$$g(x,y)=\int_0^1\int_0^s x^{s/6}y^{t/6}dtds=\frac{36\big(x^{1/6}((y^{1/6}-1)\log x-\log y)+\log y \big)}{\log x \log y (\log x +\log y )}.$$
And finally, the result is rewritten in terms of $\nabla:=\log\Delta=-\frac32[h,\cdot]$.
\begin{theorem}\label{ncconformalscalarcurvature}
For the noncommutative three tori $\nctorus$ equipped with a conformally flat  metric $g=e^{-2h}\operatorname{I}_3$, the scalar curvature $R$ is given by 
$$R=a_2(\tilde{\lap}_{0,h})=\frac{k^{-2}}{\pi^{3/2}}\Big(K(\nabla)\left(\lap (\log k)\right)+H(\nabla_{(1)},\nabla_{(2)}) \left({\textstyle{\sum}\, }\delta_j(\log k)\cdot \delta_j(\log k)\right)\Big),$$
where $\lap (x)=\sum_{j=1}^3\delta_j^2(x)$, $k=e^{h/2}$. 
The one variable function $K$ is given by
\begin{equation}\label{Kfunctionconformal}
K(s)=\frac{1-e^{s/3}}{s(e^{s/6}+e^{s/2})},
\end{equation}
and the two variable function $H$ is given by
\begin{equation}
H(s,t)
=-\frac{3  \left(\left(e^{s/3}+3\right) s \left(e^{t/3}-1\right)-\left(e^{s/3}-1\right) \left(3 e^{t/3}+1\right) t\right)}{s t (s+t) e^{\frac{1}{6} (s+t)}\left(e^{(s+t)/3}+1\right) }.
\end{equation}
\qed
\end{theorem}

The classical limit $\theta\to 0$, is obtained by taking the limits of $K(s)$ and $H(s, t)$ as $t,s\to 0$.
We obtain 
\begin{equation*}
\lim_{s\to 0}K(s)=-\frac16,\qquad 
\lim_{(s,t)\to (0,0)}H(s,t)=\frac16.
\end{equation*}
This implies that the scalar curvature $R$   approaches  the limit 
$$-\frac{k^{-2}}{24\pi^{3/2}}\sum (2\delta_j^2(h)-\delta_j(h)\delta_j(h)),$$
as $\theta\to 0$.
It matches with  the scalar curvature $-2e^{2h}\sum (-2h_{jj}+h_j^2)$ for the three torus with the metric $g=e^{-2h}(dx^2 + dy^2 + dz^2)$ up to the factor of $k^648\pi^{3/2}$ due to our convention (see Remark \ref{thediffofcoeff}).

\subsection{The Ricci density}
In this section, we shall compute the Ricci density of $\nctorus$ equipped with a conformally flat metric.
To this end,  we first need to find the term \eqref{secondtermintegralform} for $\tilde{\lap}_{1,h}$ which is anti-unitarily equivalent to the Laplacian on 1-forms.
We shall follow all the computational steps listed in the previous section to compute the scalar curvature, with one difference that the symbols are matrix valued in this case and the results will be in the matrix form. 
We start with the symbol of $\tilde{\lap}_{1,h}$.
\begin{lemma}
If we denote the symbol of $\tilde{\lap}_{1,h}$ by
$\sigma(\tilde{\lap}_{1,h})=a_{2} + a_{1} + a_{0},$
then we have
\begin{equation*}
\begin{aligned}
a_{2}=&(k^{4}\xi_{1}^{2}+k^{4}\xi_{2}^{2}+k^{4}\xi_{3}^{2})\mathrm{I}_{3},\\
a_{1}=&
\begin{pmatrix}
k^{5}\delta_{1}(k^{-1}) + k^{-1}\delta_{1}(k^{5}) & -k\delta_{2}(k^{4})k^{-1} & -k\delta_{3}(k^{4})k^{-1}\\
k^{-1}\delta_{2}(k^{4})k & k^{3}\delta_{1}(k) + k\delta_{1}(k^{3}) & 0\\
k^{-1}\delta_{3}(k^{4})k & 0 & k^{3}\delta_{1}(k) + k\delta_{1}(k^{3})  
\end{pmatrix}
\xi_{1}\\
&+
\begin{pmatrix}
k^{3}\delta_{2}(k)+k\delta_{2}(k^{3}) & k^{-1}\delta_{1}(k^{4})k & 0\\
-k\delta_{1}(k^{4})k^{-1} & k^{5}\delta_{2}(k^{-1})+k^{-1}\delta_{2}(k^{5}) & -k\delta_{3}(k^{4})k^{-1}\\
0 & k^{-1}\delta_{3}(k^{4})k & k^{3}\delta_{2}(k) + k\delta_{2}(k^{3}) 
\end{pmatrix}
\xi_{2}\\
&+
\begin{pmatrix}
k^{3}\delta_{3}(k)+k\delta_{3}(k^{3}) & 0 & k^{-1}\delta_{1}(k^{4})k\\
0 & k^{3}\delta_{3}k+k\delta_{3}(k^{3}) & k^{-1}\delta_{2}(k^{4})k\\
-k\delta_{1}(k^{4})k^{-1} & -k\delta_{2}(k^{4})k^{-1} & k^{5}\delta_{3}(k^{-1})+k^{-1}\delta_{3}(k^{5})
\end{pmatrix} 
\xi_{3},\\
a_{0}=&\sum_{1\leqslant i,j \leqslant 3}\Big(k^{-1}\delta_{i}(k^{6}\delta_{j}(k^{-1}))-k\delta_j(k^2 \delta_i(k))\Big)E_{ij}+\sum_{j=1}^{3}k\delta_j (k^2 \delta_j(k))\mathrm{I}_3.
\end{aligned}
\end{equation*}
Here $E_{ij}$'s are the matrix units.\qed
\end{lemma}

To compute  $b_2(\xi,-1)$,   we use the symbol of $\tilde{\lap}_{1,h}$ and \eqref{bjs}.  Then \eqref{secondtermintegralform1} gives the second heat trace density $a_2(\tilde{\lap}_{1,h})$.
\begin{proposition}\label{conformalalpha-1}
With notation as above, we have
\begin{equation*}
\begin{aligned}
\pi^{3/2}k^2 a_2(\tilde{\lap}_{1,h})&=\left(-\frac{1}{2}K(\nabla)\left(\lap (\log k)\right ) + T(\nabla_{(1)},\nabla_{(2)})\left({\textstyle{\sum}}\, \delta_{i}(\log k)\cdot \delta_{i}(\log k) \right)\right)  {\rm I}_{3}\\
&+\sum_{i,j=1}^{3}\left(F(\nabla)\left( \delta_{i} \delta_{j}(\log k)\right) + W(\nabla_{(1)},\nabla_{(2)})\left(\delta_{i}(\log k) \cdot \delta_{j}(\log k)\right)\right.\\
&\left.\qquad\qquad\qquad\qquad\qquad\qquad+S(\nabla_{(1)}, \nabla_{(2)})\left([\delta_{j}(\log k), \delta_{i}(\log k)]\right)\right) E_{ij},
\end{aligned}
\end{equation*}
where  $K$ is the function in \eqref{Kfunctionconformal}, and the other functions are given as follow:
\begin{align*}
F(s)&=\frac{e^{-\frac{s}2}(e^s-1)}{2(1+e^{\frac{s}3})s},\\
T(s,t)&=\frac{3s (1- e^{\frac{t}3}) (e^{\frac{2s+t}3}- e^{\frac{s+t}{3}}-e^{\frac{2 s}{3}}-1)+3t (1-e^{\frac{s}3}) (e^{ \frac{s+2 t}3}+e^{\frac{s}3}+e^{\frac{t}3}-1)}{s t(s+t)e^{\scriptstyle \frac{3s+t}6}   (e^{\frac{(s+t)}3}+1 ) },\\
W(s,t)&= \frac{6    (e^{\frac{s+t}{3}}+e^{\frac{2 (s+t)}{3}}+1 )  (s e^{\frac{s+t}{3}}-e^{\frac{s}3} (s+t)+t )}{s t (s+t) e^{\frac{s+t}{2}} (e^{\frac{s+t}{3}}+1 ) },\\
S(s,t)&=\hspace{-0.1cm}\frac{3 s  (e^{\frac{t}3}-1 )  (2 e^{\frac{s+t}{3}}+e^{\frac{ 2s+2t}{3}}-e^{\frac{2 s+t}{3} }+1 )\hspace{-0.1cm}- 3t(e^{\frac{s}3}-1 )   (2 e^{\frac{s+2t}{3}}+e^{\frac{2s+3t}{3}}\hspace{-0.1cm}-e^{\frac{s+t}3}+ e^{\frac{t}3} ) }{s t (s+t) e^{\frac{1}{2} (s+t)}(e^{\frac{s+t}{3}}+1 ) }.
\end{align*}\qed
\end{proposition}
Using  definitions \ref{ncscalar} and \ref{ncricci},  Theorem \ref{conformalalpha-1}, and Proposition \ref{ncconformalscalarcurvature}, we can compute the Ricci density of the noncommutative three tori $\nctorus$ equipped with a conformally flat  metric $g=e^{-2h}\operatorname{I}_3$.
\begin{theorem}\label{ncconformalRiccicurvature}
The Ricci density of $\nctorus$ equipped with the conformally flat metric $g=e^{-2h}\operatorname{I}_3$ is given by
\begin{equation*}
\begin{aligned}
\ricden=&\pi^{-\frac32}k^{-2}\left(\frac{3}{2}K(\nabla)\left(\lap (\log k)\right ) + (H-T)(\nabla_{(1)},\nabla_{(2)})\left(\textstyle{\sum}\delta_{\ell}(\log k)\cdot\delta_{\ell}(\log k) \right)\right)  {\rm I}_{3}\\
&-\pi^{-\frac32}k^{-2}\sum\Big(F(\nabla)\left( \delta_{i} \delta_{j}(\log k)\right) + W(\nabla_{(1)},\nabla_{(2)})\left(\delta_{i}(\log k)\cdot \delta_{j}(\log k)\right)\\
&\qquad \qquad\qquad\quad +S(\nabla_{(1)}, \nabla_{(2)})\left([\delta_{j}(\log k), \delta_{i}(\log k)]\right)\Big) E_{ij}.
\end{aligned}
\end{equation*}
Here $k=e^{h/2}$, and $\lap(a)=\sum \delta_j^2(a)$ denotes the flat Laplacian.\qed
\end{theorem}

\begin{remark}
To check the result with the commutative case, we need to find the following limits:
\begin{equation*}
\lim_{s\to 0}F(s)=\frac14,
\qquad \lim_{(s,t)\to(0,0)}T(s,t)=-\frac{1}{3}, \qquad \lim_{(s,t)\to(0,0)}W(s,t)=\frac12.
\end{equation*}
Since in the commutative case the commutator term $[\delta_{j}(\log k), \delta_{i}(\log k)]$ on which $S$ acts, automatically vanishes,
we find that the $(i,j)^{th}$ entry of the Ricci density for $\theta=0$ is given by
\begin{equation}\label{ricconformalasthetagoestozero}
-\frac{k^{-2}}{8\pi^{3/2}}\left(\delta_{ij}\big(\sum_{\ell=1}^{3}\delta^2_{\ell}(h)-\delta_{\ell}(h)^2\big)+\delta_{i}(h)\delta_{j}(h)+\delta_{i}(\delta_{j}(h))\right),
\end{equation}
\noindent where the $\delta_{ij}$ denotes the Kronecker delta. 
On the other hand, a direct computation in the commutative case for the metric $g=e^{-2h}\operatorname{I}_3$ gives the $(i,j)^{th}$ component of the Ricci operator as
\begin{equation*}
e^{2h}\left(\delta_{ij}(\sum_{\ell=1}^3 h_{\ell\ell}-h_\ell{}^2)+h_ih_j+h_{ij}\right),
\end{equation*} 
which matches with the corresponding Ricci density  in \eqref{ricconformalasthetagoestozero} after taking into the account the Remark \ref{thediffofcoeff}.
\end{remark}

\section{Ricci density for non-conformal perturbations} \label{nonconformalmetrics}

In this section we shall compute the Ricci curvature for a metric on the noncommutative three torus which is an analogue of the metric  
\begin{equation}\label{metric_2}
e^{-2h}(dx^2+dy^2)+dz^2,
\end{equation}
for some $h\in C^\infty(\mathbb{T}^3)$ in the classical case.
The inner products on functions, 1-forms and 2-forms for a torus equipped with this metric are given as follows:
\begin{equation*}
\langle f_1, f_2 \rangle = \int_{\mathbb{T}^3} f_1\overline{f_2}e^{-2h}dx dy dz,
\end{equation*}
for all $f_1,f_2 \in \Omega^0(\mathbb{T}^3),$
\begin{equation*}
\langle \alpha, \beta\rangle= \int_{\mathbb{T}^3} \left(\alpha_1\overline{\beta_1} + \alpha_2\overline{\beta_2} + \alpha_3\overline{\beta_3}e^{-2h} \right)dxdydz, 
\end{equation*}
for all $\alpha=(\alpha_1,\alpha_2,\alpha_3),\beta=(\beta_1,\beta_2,\beta_3)\in \Omega^1(\mathbb{T}^3)$, and
\begin{equation*}
\langle \xi, \eta\rangle = \int_{\mathbb{T}^3} \left( \xi_1 \overline{\eta_1}e^{2h} + \xi_2\overline{\eta_2}+\xi_3\overline{\eta_3}\right) dx dy dz,
\end{equation*}
for all $\xi = (\xi_1,\xi_2, \xi_3), \eta = (\eta_1, \eta_2, \eta_3)\in \Omega^2(\mathbb{T}^3)$.

Let $k=e^h$ for $h\in \Ai$.
Motivated by the classical case,  we denote by $\mathcal{H}_h^{(0)}$ the Hilbert space given by the GNS construction of $\A$ with respect to the positive linear functional 
\begin{equation*}
\varphi_0(a)=\varphi(ak^{-2}).
\end{equation*}
For 1-forms, we denote by $\mathcal{H}_h^{(1)}$ the Hilbert space, which is the completion of $\Omega^1\nctorus$ with respect to the inner product given by
\begin{equation*}
\langle a, b\rangle= \varphi\left(b_1^{*}a_1+b_2^{*}a_2+b_3^{*}a_3k^{-2}\right).
\end{equation*} 
For 2-forms, we denote by $\mathcal{H}_h^{(2)}$ the Hilbert space, which is the completion of $\Omega^2\nctorus$ with respect to the inner product given by 
\begin{equation*}
\langle a, b\rangle = \varphi\left(b_1^{*}a_1 k^2 + b_2^{*}a_2 + b_3^{*}a_3\right).
\end{equation*}

We also need adjoints of de Rham differentials \eqref{deRhamcomplex} with respect to the given metric. 
It can be shown that the adjoint of $d_0$ is given by 
\begin{equation*}
d_0^{*}: b\mapsto (-i)(\delta_1(b_1)k^2+\delta_2(b_2)k^2+\delta_3(b_3)-b_3k^{-2}\delta_3(k^2)),\quad b=(b_1, b_2, b_3)\in \Omega^1\nctorus. 
\end{equation*}
Similarly, the adjoint of  $d_1$ acting on an element $a=\left(a_1, a_2, a_3\right)\in \Omega^2\nctorus$ is given by
\begin{equation*}
d_1^*:a\mapsto \left(i\delta_2(a_1k^2) + i\delta_3(a_3), i\delta_3(a_2)-i\delta_1(a_1k^2), -i\delta_2(a_2)k^2-i\delta_1(a_3)k^2\right).
\end{equation*}

To compute the spectral densities of the Laplacians for these metrics, we will follow the steps presented in section \ref{Pseudoandspectral}. 
By a homogeneity argument, again, the computation  of contour integral can be bypassed by setting $\lambda=-1$;
\begin{equation*}
\frac{1}{(2\pi)^3}\int_{\mathbb{R}^3}\frac{1}{2\pi i}\int b_2(\xi,\lambda)d\lambda d\xi=\frac{1}{8\pi^{7/2}}\int_{\mathbb{R}^3}b_2(\xi,-1)d\xi.
\end{equation*}
Then we have integrals in $\xi$ variable where the dependence of the integrand comes from the powers of $b_0(\xi,-1)=(1+a_2(\xi))^{-1}$ and $\xi_j$. 
To compute these integrals, we first apply a change of variables,
\begin{equation}\label{substitution}
\xi_1=\sqrt{u(1+\eta^2)}\cos\theta,\quad \xi_2=\sqrt{u(1+\eta^2)}\sin\theta,\quad \xi_3=\eta,
\end{equation}  
where the domain of the new variables $(u,\eta,\theta)$ is given by
$$ u\in[0,+\infty),\quad \eta\in(-\infty,+\infty),\quad \theta\in[0,2\pi).$$
The Jacobian of this substitution is $\frac12(1+\eta^2),$ 
and this substitution decomposes $b_0$  to $(1+\eta^2)^{-1}$ multiplied by a noncommutative part which depends only on $u$. More precisely 
\begin{equation*}
b_0(\xi,-1)=(1+k^2\xi_1^2+k^2\xi_2^2+\xi_3^2)^{-1}=(1+\eta^2+u(1+\eta^2)k^2)^{-1}
=\frac{1}{1+\eta^2}b_0(u).
\end{equation*}
Here we denoted $(1+uk^2)^{-1}$ by $b_0(u)$. 
As a result, after applying the substitution, each term of $b_2$ ends up with a triple integral whose two variables $(\eta,\theta)$ can be separated and integrated, without involving any noncommutative terms.
For instance,
\begin{equation*}
\begin{aligned}
&\int_{\mathbb{R}^3} \xi_2^4 \xi_3^2 b_0^3(\xi,-1)\delta_3(k^2)b_0(\xi,-1)\delta_3(k^2)b_0(\xi,-1) d\xi\\
&=\int_0^\infty \int_{-\infty}^{\infty}\int_{0}^{2\pi}  \frac{ u^2\eta^2(1+\eta^2)^2 \sin^4\theta}{(1+\eta^2)^5} b_0^3(u)\delta_3(k^2)b_0(u)\delta_3(k^2)b_0(u) \frac{1}{2}(1 + \eta^2) d\eta d\theta du\\
&=\left(\int_{-\infty}^{\infty} \frac{\eta^2 }{2(1+\eta^2)^2} d\eta\right)  
\left(\int_{0}^{2\pi}  \sin^4\theta d\theta \right)
\int_0^\infty   u^2 b_0^3(u)\delta_3(k^2)b_0(u)\delta_3(k^2)b_0(u) du\\
&= \frac{3 \pi^2}{16} \int_0^\infty   u^2 b_0^3(u)\delta_3(k^2)b_0(u)\delta_3(k^2)b_0(u) du.
\end{aligned}
\end{equation*}

Applying the substitution and integrating out the $\eta$ and $\theta$ variables,  we end up with sums of $u$ integrals in one of the following forms:
\begin{equation*}
\int_0^\infty b_0(u)^{m_0}\rho_1b_0(u)^{m_1}\rho_2\cdots \rho_p b_0(u)^{m_p} u^{-2+\sum m_j}du,
\end{equation*} 
or
\begin{equation*}
\int_0^\infty b_0(u)^{m_0}\rho_1b_0(u)^{m_1}\rho_2\cdots \rho_p b_0(u)^{m_p} u^{-3+\sum m_j}du.
\end{equation*} 
Here we need Proposition \ref{leschrearrangemnet} for  
\begin{equation*}
\begin{aligned}
&f_0(x):=x^{\sum m_j - \nu}(1+x)^{-m_0},\\
&f_j(x):=(1+x)^{-m_j},\quad j=1,...,p,
\end{aligned}
\end{equation*}
and   $a=2h$.  Here $\nu$ is equal to $2$ or $3$.
We then get the following version of the rearrangement lemma.
\begin{corollary}\label{rearrangenoncon}
Let $b_0=(1+uk^2)^{-1}$, $\rho_j\in \Ai$, $m_j\in \mathbb{Z}$, for $j=0,1,2,...,p$, and $\Delta(x)=k^{-2}xk^2$.
Then
\begin{equation*}
\begin{aligned}
&\int_{0}^{\infty}b_0(u)^{m_0}\rho_1 b_0(u)^{m_1}\rho_2b_0(u)^{m_2}\cdots \rho_lb_0(u)^{m_p} u^{(-\nu + \sum m_j)}du\\
&=k^{2(-\sum_{j=0}^{p}m_j+\nu - 1)}F^{[\nu]}_{m_0,m_1,...,m_p}(\Delta_{(1)}, \Delta_{(2)},...,\Delta_{(p)})(\rho_1\cdot \rho_2\cdots \rho_p),
\end{aligned}
\end{equation*}
where
\begin{equation*}
F^{[\nu]}_{m_0,m_1,...,m_p}(s_1,s_2,...,s_p)=
\int_{0}^{\infty}(1+u)^{-m_0}\prod_{j=1}^{p}\Big( u\prod_{h=1}^{j}s_{h} + 1\Big)^{-m_j}u^{(\sum m_j-\nu)}du.
\end{equation*}
\qed
\end{corollary}
For instance,
\begin{align*}
&F^{[2]}_{1,1}(s_1)=\frac{\log (s_1)}{s_1-1},\\
&F^{[3]}_{2,1}(s_1)=\frac{s_1 (\log (s_1)-1)+1}{(s_1-1)^2},\\
&F^{[2]}_{1,1,1}(s_1,s_2)=\frac{(s_1 s_2-1) \log (s_1)-(s_1-1) \log (s_1 s_2)}{(s_1-1) s_1 (s_2-1) (s_1 s_2-1)},\\
&F^{[3]}_{1,1,1}(s_1,s_2)=\frac{-s_1 s_2 \log (s_1)+s_1 s_2 \log (s_1 s_2)-s_2 \log (s_1 s_2)+\log (s_1)}{(s_1-1) (s_2-1) (s_1 s_2-1)}.
\end{align*}

We also need the following result from \cite[Section 6.1]{Connes-Moscovici2014}, according to which  we find the formula  
\begin{equation}\label{translatetologk}
\begin{aligned}
&k^{-1}\delta_j(k)=f(\Delta)\left(\delta_j(\log k)\right),\\
&k^{-1}\delta_j^2(k)=f(\Delta)(\delta_j^2(\log k))+2 g(\Delta_{(1)},\Delta_{(2)})(\delta_j(\log k)\cdot\delta_j(\log k)),
\end{aligned}
\end{equation}
where
$$f(x)=\int_0^1 x^{s/2}ds=\frac{2(\sqrt{x}-1)}{\log x },$$
$$g(x,y)=\int_0^1\int_0^s x^{s/2}y^{t/2}dtds=\frac{4\big(\sqrt{x}((\sqrt{y}-1)\log x -\log y )+\log y \big)}{\log x \log y (\log x +\log y )}.$$
Now we can start computing the Laplacians and their spectral densities.

\subsection{Scalar curvature}\label{section_scalar_curvature}
In this section, we first find the Laplacian on functions $\lap_{0,h}$ for the given metric and its anti-unitary equivalent differential operator $\tilde{\lap}_{0,h}$.
Then we use its symbol and its resolvent expansion to find the scalar curvature. 

The Laplacian on functions $\lap_{0, h}:\Ai\to \Ai$ for the metric \eqref{metric_2}, which is given by $\lap_{0,h} = d_0^*d_0$, computes as
\begin{equation*}
\lap_{0,h}(a)=\delta _1^2(a)k^2+\delta _2^2(a)k^2+\delta_3(\delta_3(a)k^{-2})k^2.
\end{equation*}
We define the map $R_{0,k}: \mathcal{H}_{0,0}\rightarrow \mathcal{H}_{0,h}$ by $R_{0,k}a=ak$,
for all $a\in\A$.
It is not hard to see that $R_{0,k}$ is an isometry from $\mathcal{H}_{0,0}$ to $\mathcal{H}_{0,h}$.
That is,
$
\langle R_{0,k}a, R_{0,k}b\rangle_{0,h} = \langle a, b\rangle_{0,0}.
$
Hence, the Laplacian on functions $\lap_{0,h}$ for the metric \eqref{metric_2} is anti-unitary equivalent to the differential operator $\left(R_{0,k}J \right)^*\lap_{0,h}R_{0,k}J$ on $\Omega^0\mathbb{T}^3_\theta$, 
which we denote by $\tilde{\lap}_{0,h}$.

\begin{lemma}
The homogeneous components of the symbol $\sigma(\tilde{\lap}_{0,h})$ are:
\begin{equation*}
\begin{aligned}
a_2&=k^2\xi_1^2+k^2\xi_2^2+\xi_3^2,\\
a_1&=2k\delta_1(k)\xi_1+2k\delta_2(k)\xi_2+\big(k^{-1}\delta_3(k)-\delta_3(k)k^{-1}\big)\xi_3,\\
a_0&=k\delta_1^2(k)+k\delta_2^2(k)+k^{-1}\delta_3^2(k)-\delta_3(k)k^{-2}\delta_3(k)-k^{-1}\delta_3(k)k^{-1}\delta_3(k).
\end{aligned}
\end{equation*}
\end{lemma}
\begin{proof}
It can be readily checked that the operator $\tilde{\lap}_{0,h}$, on the elements of $\Ai$, is given by  
\begin{equation}\label{2nd_tilde_lap_0}
\begin{aligned}
\tilde{\lap}_{0,h}(a)
&=k^2\delta _1^2(a)+k^2\delta _2^2(a)+\delta _3^2(a)\\
&+2 k\delta _1(k)\delta _1(a)+2 k\delta _2(k)\delta _2(a)-k^{-1}\delta _3\left(k^2\right)k^{-1}\delta _3(a)+2 k^{-1}\delta _3(k)\delta _3(a)\\
&+k^{-1}\delta _3^2(k)a+k\delta _1^2(k)a+k\delta _2^2(k)a-k^{-1}\delta _3\left(k^2\right)k^{-2}\delta _3(k)a.
\end{aligned}
\end{equation}
 Then the symbol is given by replacing $\delta_j$ by $\xi_j$.
\end{proof}

The scalar curvature of $\nctorus$ equipped with the metric \eqref{metric_2} is defined as in  Definition \ref{ncscalar}.  Similar to the conformal case it is given by  \eqref{scalarcurvature2} where  $b_2$ is the second term in the symbol of the parametrix of $\tilde{\Delta}_{0,h}$ for this metric.
The computation then shows that we have:
\begin{theorem}\label{2nd_scalar_thm}
If the noncommutative 3-torus $\nctorus$ is equipped with the non-conformal metric \eqref{metric_2}, then its scalar curvature $R$ is given by
\begin{equation*}
\begin{aligned}
\pi^{3/2}a_2(\tilde{\lap}_{0,h}) =& K_1(\nabla)(\delta_1^{2}(h)+\delta_2^{2}(h))+H_1(\nabla_{(1)}, \nabla_{(2)})(\delta_1(h)\cdot\delta_1(h)+\delta_2(h)\cdot\delta_2(h))\\
&+k^{-2}K_2(\nabla)(\delta_3^2(h)) +k^{-2}H_2(\nabla_{(1)}, \nabla_{(2)})(\delta_3(h)\cdot\delta_3(h)),
\end{aligned}
\end{equation*}
where 
\begin{equation*}
\begin{aligned}
K_1(s)=&\frac{e^{\frac{s}2}(2e^s-se^s-2-s)}{4s(e^s-1)^2},\\
 K_2(s)=& \frac{1- e^{2s}+2 s e^s}{4 s e^{\frac{s}2} (1- e^s)^2},\\
H_1(s,t)=&\frac{1}{e^{-\frac{s+t}{2}} (e^s-1 ) s  (e^t-1 ) t  (e^{s+t}-1 )^2 (s+t)}\Big( e^s   (e^t-1 )^2 s^2- e^t(e^s-1 )^2  t^2\\
&\qquad\qquad\qquad\qquad -  (e^s-e^t )  (e^{s+t}-1 )s t + (1-e^s )  (e^t-1 )  (e^{s+t}-1 ) (t-s) \Big),\\
H_2(s, t)=&\frac{1}{4 e^{\frac{1}{2} (s+t)} (e^s-1 )   (e^t-1 )  (e^{s+t}-1 )^2st  (s+t)}\Big((e^t-1 )^2   (e^{s+t}-3 e^{2 s+t}-e^s-1 ) s^2\\
& + (e^s-1 )^2    (e^{s+2t}+e^{s+3 t}-e^{2t}+3e^t )t^2 -2(e^s-1 )(e^t-1)(e^{2 (s+t)}-1 )(s-t) \\
&  + (e^{s+t}-1) (4 e^{s+t}+e^{2 s+t}-5 e^{s+2 t}+e^{2 s+2 t}+e^s-5 e^t+2 e^{2 t}+1 )s t \Big).
\end{aligned}
\end{equation*}
\qed
\end{theorem}

We can get the classical scalar curvature in  the limit $\theta \rightarrow 0$, which is obtained by taking the limits of the above functions as $s,t\rightarrow 0$.
We have 
\begin{equation*}
\lim\limits_{(s,t)\to (0,0)} H_1(s, t)=0,\quad\lim\limits_{(s,t)\to (0,0)}  H_2(s, t)=\frac18,\quad\lim_{s\to 0} K_1(s)= -\frac{1}{24},\quad\lim_{s\to 0} K_2(s)= -\frac{1}{12}.
\end{equation*}
Therefore, when $\theta\to 0$, the scalar curvature approaches to
\begin{equation*}
-\frac1{48\pi^{3/2}}\left(2\delta_1^{2}(h)+2\delta_2^{2}(h)+4e^{-2h}\delta_3^2(h)-6e^{-2h}\delta_3(h)^2\right),
\end{equation*}
which is $\frac{e^{-2h}}{48\pi^{3/2}}$  multiple of the scalar curvature,
$2e^{2h}(h_{11}+h_{22}) + 4h_{33} -6(h_3)^2$, in the commutative case. 
This matches with our normalization of the scalar curvature density.

\begin{remark}\label{resultsfromtwodimconformal}
Comparing  the functions $K_1$ and $H_1$  with the corresponding  functions  $K$ and  $H$ found in 
  \cite{Connes-Moscovici2014,Fathizadeh-Khalkhali2013}   for the spectral densities of the Laplacian  $k\partial^{\ast}\partial k$ reveals that 
\begin{equation}\label{functionslimitnonconformal}
K_1(s)=-\frac18 K(s),\qquad H_1(s,t)=-\frac18  H(s,t).
\end{equation}
The factor $-\frac18$ is the result of the use of two different normalizations.  
In the rest of this section we shall look for a clarification of why such a relation \eqref{functionslimitnonconformal} should be true. 

First note that the Laplacian on functions $\tilde{\lap}_{0,h}^{(1)}$, given in \eqref{2nd_tilde_lap_0}, is the sum of two Laplacians when we assume that $\delta_3(k)=0$;
\begin{equation*}
\tilde{\lap}_{0,h}=\tilde{\lap}_{0,h}^{(1)}\otimes 1 + 1\otimes\tilde{\lap}_{0,h}^{(2)},
\end{equation*}
where 
\begin{equation*}
\tilde{\lap}_{0,h}^{(1)}=\sum_{i=1}^2k^2\delta _i^2(a)+2 k\delta _i(k)\delta _i(a)+k\delta _i^2(k)a,\qquad \tilde{\lap}_{0,h}^{(2)}=\delta_{3}^{2}.
\end{equation*}
The operator $\tilde{\lap}_{0,h}^{(1)}$ is equal to the operator $k\partial^{\ast}\partial k$, which is anti-unitarily equivalent to the Laplacian on $C^{\infty}(\mathbb{T}_{\theta}^{2})$ in  \cite[Section 4.1]{Fathizadeh-Khalkhali2013} when the complex structure is given by $\tau = i$, namely $\tau_1=0$, $\tau_2=1$.
The operator $\tilde{\lap}_{0,h}^{(2)}$ is the Laplacian of $\mathbb{T}^1$ with flat metric.
Then,  the local spectral invariants of $\tilde{\lap}_{0,h}$ are related to those of $\tilde{\lap}_{0,h}^{(1)}$ and $\tilde{\lap}_{0,h}^{(2)}$ as we discuss next.

Let $P$ and $Q$ be  two elliptic second order positive differential operators on $C({ \mathbb{T}}^d_\theta)$ and $C({ \mathbb{T}}^{d'}_{\theta'})$ respectively.
Then $P\otimes 1 +1\otimes Q$ forms a  positive second order elliptic differential operator on $C({ \mathbb{T}}^d_\theta)\otimes C({ \mathbb{T}}^{d'}_{\theta'})$.
Moreover, for any $t>0$ and $a\in A_\theta$ and $b\in A_{\theta,}$ we have
$$\Tr(a\otimes b e^{-t(P\otimes 1 +1\otimes Q)})=\Tr(ae^{-tP})\Tr(be^{-tQ}), \quad a\in C^\infty({ \mathbb{T}}^d_\theta),\, b\in C^\infty({ \mathbb{T}}^{d'}_{\theta'}),  \, \, t>0.$$ 
This not only gives a relations between the coefficients of asymptotic expansions  as $t\to 0^+$, but also it provides a relation among the densities of these coefficients.
In other words if 
$$\Tr(ae^{-tP})\sim \sum_{n=0}^\infty t^{n-\frac{d}{2}} \varphi_\theta(a a_n(P)), \qquad \Tr(be^{-tP})\sim \sum_{m=0}^\infty t^{m-\frac{d'}2} \varphi_{\theta'}(b a_m(Q)),$$
where $\varphi_{\theta}$ and $\varphi_{\theta'}$ is the tracial state on $C^\infty({ \mathbb{T}}^d_\theta)$ and $ C^\infty({ \mathbb{T}}^{d'}_{\theta'})$, respectively,
then 
\begin{equation*}
\begin{aligned}
\operatorname{Tr}(a\otimes b e^{-t(P\otimes 1 +1\otimes Q)})
&=\sum_{n=0}^\infty \sum_{m=0}^\infty t^{m+n-\frac{d'}2-\frac{d}{2}} \varphi_\theta(a a_n(P))  \varphi_{\theta'}(b a_m(Q))\\
&=\sum_{l=0}^\infty t^{l-\frac{d'+d}{2}} \varphi_\theta\otimes \varphi_{\theta'}\left(a\otimes b \Big(\sum_{l=m+n}a_n(P)\otimes a_m(Q)\Big)\right).
\end{aligned}
\end{equation*}
In our case, we have 
\begin{equation*}
a_2(\tilde{\lap}_{0,h})=a_2(\tilde{\lap}_{0,h}^{(1)})\otimes a_0(\tilde{\lap}_{0,h}^{(2)})+a_0(\tilde{\lap}_{0,h}^{(1)})\otimes a_2(\tilde{\lap}_{0,h}^{(2)}).
\end{equation*}
However, since $\sigma(\tilde{\lap}_{0,h}^{(2)})=\xi^2$,  we have  $a_2(\tilde{\lap}_{0,h}^{(2)})=0$ and $a_0(\tilde{\lap}_{0,h}^{(2)})=\sqrt{\pi}$.
Thus 
\begin{equation*}
a_2(\tilde{\lap}_{0,h})=\sqrt{\pi}a_2(\tilde{\lap}_{0,h}^{(1)}).
\end{equation*}
This is the main reason for why the functions of two dimensional noncommutative two torus with conformally flat metric emerge in the formulas for the  noncommutative three torus with non-conformal metric \eqref{metric_2}.  On the other hand, we note that the functions $K_2$ and $H_2$ in Theorem \ref{2nd_scalar_thm} are new and do not 
seem to be related to functions for the noncommutative two torus. 
\end{remark}

\subsection{Laplacian on 1-forms and the Ricci density}
In this section, after finding the Laplacian on 1-forms on $\nctorus$ equipped with the metric \eqref{metric_2}, we compute its second heat trace density.
Combining with the results from the previous section, we shall then compute the Ricci density of this metric.

Recall that exterior derivative on  1-forms is given by
\begin{equation*}
d_1(a_1,a_2,a_3) = \left(i\delta_1(a_2)-i\delta_2(a_1), i\delta_2(a_3)-i\delta_3(a_2), i\delta_1(a_3)-i\delta_3(a_1)\right),
\end{equation*}
and  hence its formal adjoint with respect to the metric is
\begin{equation*}
d_1^{*}(a_1,a_2,a_3)=\left(i\delta_2(a_1k^2)+i\delta_3(a_3),i\delta_3(a_2)-i\delta_1(a_1k^2), -i\delta_2(a_2)k^2-i\delta_1(a_3)k^2 \right).
\end{equation*}
Thus, the Laplacian on 1-forms $\lap_{1,h}$ computes as  
\begin{eqnarray*}
&\lap_{1,h}(a_1,a_2,a_3)=d_0d_0^*(a_1,a_2,a_3) + d_1^*d_1(a_1,a_2,a_3)=\\
&\Big(\delta _1(\delta _1( a_1) k^2)+\delta_2(\delta _2( a_1) k^2)+\delta _3^2( a_1)+\delta _2( a_2) \delta _1(k^2)-\delta _1(a_2) \delta _2(k^2)-\delta _1( a_3 k^{-2} \delta _3(k^2)),\\
&\delta _1( a_1) \delta _2 (k^2 )-\delta _2( a_1) \delta _1 (k^2 )+\delta _1 (\delta _1( a_2) k^2 )+\delta _2 (\delta _2( a_2) k^2 )+\delta _3^2( a_2) -\delta _2 ( a_3 k^{-2}  \delta _3 (k^2 ) ),\\
&\delta _1( a_1) \delta _3 (k^2 )+\delta _2( a_2)\delta _3 ( k^2 )+\delta _1^2( a_3)k^2+\delta _2^2( a_3)k^2+\delta _3 (\delta_3( a_3 k^{-2}) k^2 )\ \Big).
\end{eqnarray*}
\begin{lemma}
The Laplacian on 1-forms ${\lap}_{1,h}$ is anti-unitary equivalent to a differential operator $\tilde{\lap}_{1,h}$ whose symbol is the sum of the homogeneous components given by
\begin{equation*}
a_2= (k^2 \xi _1^2+k^2 \xi _2^2+\xi _3^2)\operatorname{I}_3,
\end{equation*}
\begin{equation*}
a_1 = 
\left (
\begin{array}{ccc}
 \delta _1 (k^2 )\xi_1+\delta _2 (k^2 )\xi _2  & \delta _1 (k^2 )\xi _2-\delta _2 (k^2 )\xi _1 & -\delta _3 (k^2 )k^{-1} \xi _1 \\
 \delta _2 (k^2 )\xi _1- \delta _1 (k^2 )\xi _2 & \delta _1 (k^2 )\xi _1+\delta _2 (k^2 )\xi _2  & -\delta _3 (k^2 )k^{-1} \xi _2 \\
k^{-1}\delta _3 (k^2 ) \xi _1 & k^{-1}\delta _3 (k^2 ) \xi _2 & 2 k\sum\limits_{i=1}^2\delta _i(k) \xi _i+ [k^{-1}, \delta _3(k)] \xi _3\\
\end{array}
 \right),
\end{equation*}
\begin{equation*}
a_0=
 \left(
\begin{array}{ccc}
 0 & 0 & -\delta_1(\delta_3(k^2)k^{-1}) \\
 0 & 0 &  -\delta_2(\delta_3(k^2)k^{-1}) \\
 0 & 0 & k\delta_1^2(k) + k\delta_2^2(k) +k^{-1}\delta_3(k^2\delta_3(k^{-1})) \\
\end{array}
 \right).
\end{equation*}
\end{lemma}
\begin{proof}
Denote by $R_{1,k}: \mathcal{H}_{1,0}\rightarrow \mathcal{H}_{1,h}$ the operator defined as 
\begin{equation*}
R_{1,k}\left(b_1,b_2,b_3\right) = \left( b_1,b_2, b_3k\right).
\end{equation*}
We notice that $R_{1,k}: \mathcal{H}_{1,0}\rightarrow \mathcal{H}_{1,h}$ is an isometry from $\mathcal{H}_{1,0}$ to $\mathcal{H}_{1,h}$.
Thus $\lap_{1,h}$ is anti-unitary equivalent to $\tilde{\lap}_{1,h}=\left( R_{1,k}J\right)^*\lap_{1,h}R_{1,k}J$ which is given by the formula
\begin{eqnarray*}
&\tilde{\lap}_{1,h}(a_1,a_2,a_3)=\\
&\Big(\delta_1(k^2\delta_1(a_1)) +\delta_2(k^2\delta_2(a_1))+ \delta _3^2(a_1) +\delta_1(k^2)\delta_2(a_2)-\delta_2(k^2)\delta_1(a_2) -\delta_1(\delta_3(k^2)k^{-1}a_3),\\
& \delta_2(k^2)\delta_1(a_1)-\delta_1(k^2)\delta_2(a_1)+\delta_1(k^2\delta_1(a_2))  +\delta_2(k^2\delta_2(a_2))+ \delta_3^2(a_2)  -\delta_2(\delta_3(k^2)k^{-1}a_3),\\
&k^{-1}\delta_3(k^2)\delta_1(a_1)+k^{-1}\delta_3(k^2)\delta_2(a_2)  +k\delta_1^2(k a_3) + k\delta_2^2(ka_3) +k^{-1}\delta_3(k^2\delta_3(k^{-1}a_3))  \Big).
\end{eqnarray*}
This proves the lemma.
\end{proof}

Then computation can be carried out to compute $a_2(\tilde{\lap}_{1,h})$, and the final result is given in the following proposition.
In this proposition to make the formulas concise, we shall use the notation 
$$F^\nabla(\rho):=F(\nabla)(\rho),\qquad F^\nabla(\rho_1\cdot \rho_2):=F(\nabla_{(1)},\nabla_{(2)})(\rho_1\cdot \rho_2),$$
for a given function $F$ with one or two variables. 
\begin{proposition}
The second density of the heat trace for the operator $\tilde{\lap}_{1,h}$ is given by 
\begin{equation*}
\begin{aligned}
\pi^{\frac32}a_2(\tilde{\lap}_{1,h})&=\left(K_{22}^\nabla(\delta_2^2(h))+2W_{22}^\nabla(\delta_2(h)^2) + k^{-2} {K}_3^\nabla(\delta_3^2(h))+k^{-2}H_3^\nabla(\delta_3(h)^2)\right)  E_{11}\\
&+\left(K_{11}^\nabla(\delta_1^2(h))+2W_{11}^\nabla(\delta_1(h)^2)+ k^{-2} {K}_3^\nabla(\delta_3^2(h))+k^{-2}H_3^\nabla(\delta_3(h)^2)\right)  E_{22}\\
&+\left(K_1^\nabla(\delta_1^2(h)+\delta_2^2(h))+H_1^\nabla(\delta_1(h)^2+\delta_2(h)^2)+k^{-2}H_4^\nabla(\delta_3(h)^2)\right) E_{33}\\
&+\sum k^{-c(i,j)}\left(K_{ij}^\nabla(\delta_i\delta_j(h))+S_{ij}^\nabla([\delta_i(h),\delta_j(h)])+W_{ij}^\nabla(\{\delta_i(h),\delta_j(h)\})\right) E_{ij}.
\end{aligned}
\end{equation*} 
Here $[\delta_i(h),\delta_j(h)]$ and $\{\delta_i(h),\delta_j(h)\}$ denote the commutator and anti-commutator. The functions are given as the entries of the following matrices.
$${\bf{K}}=\frac{1}{4s(e^s-1)}\begin{pmatrix}
\frac{e^{2s}-2se^s-1}{e^s-1} 		&	0								&	 (s -1)e^{\frac{s}{2}}+e^{-\frac{s}{2}}	\\ 
0									& \frac{e^{2s}-2se^s-1}{e^s-1}		&	 (s -1)e^{\frac{s}{2}}+e^{-\frac{s}{2}}	\\ 
e^s-s-1								&	e^s-s-1							&	\frac{1-e^{2s}+se^{2s}+s}{e^{\frac{s}2}(e^s-1)}
\end{pmatrix},
$$
$${\bf{S}}(s,t)=\begin{pmatrix}
0		&	1		&	 \frac12{e^{-\frac{s+t}{2}}}	\\ 
1		& 0		   &	\frac12{e^{-\frac{s+t}{2}}} \\ 
\frac12	&	\frac12	&	0
\end{pmatrix}S_1(s,t),$$
where 
$$S_1(s,t)=\frac{1}{2 s t}-\frac{(e^s-1)^2 e^t t+e^s s (e^t-1)^2}{2 s t (e^s-1) (e^t-1) (e^{s+t}-1)}.$$
Also,
$${\bf{W}}(s,t)=\begin{pmatrix}
\frac12\cosh(\frac{s+t}2)				&	0									&	\frac{e^{-s-t}-1}{4}	\\ 
0										&	\frac12\cosh(\frac{s+t}2) 			&	\frac{e^{-s-t}-1}{4}	\\ 
\frac12\sinh(\frac{s+t}2)				&	\frac12\sinh(\frac{s+t}2) 			&	\frac{W_{33}(s,t)}{H_1(s,t)}
\end{pmatrix}H_1(s,t).$$
Here, $H_1$ is the function from Theorem \ref{2nd_scalar_thm}. 
The function $W_{33}$, together with the remaining functions, are given below:
\begin{equation*}
\begin{aligned}
&W_{33}(s, t)=\frac{1}{16 e^{\frac{s+t}2} (e^s-1 )   (e^t-1 )  (e^{s+t}-1 )^2st  (s+t)}\times\\
&\Big((e^t-1)^2 (1-4e^s-e^{2s}-e^{s+t}-4e^{2s+t}+e^{3s+t}) s^2 +2(e^{s}+1)(e^{t}+1)(e^{s+t}-1)(e^s-e^t) s t\\
&  - (e^s-1)^2 (1\hspace{-0.1cm}-4e^t\hspace{-0.1cm}-e^{2t}-e^{s+t}-4e^{2t+s}+e^{3t+s}) t^2 -4(e^s-1)(e^t-1)(e^{2(s+t)}-1 )(s-t)  \Big),
\end{aligned}
\end{equation*}
\begin{equation*}
\begin{aligned}
H_3(s, t)&=\frac{1}{4 e^{s} (e^s-1 )   (e^t-1 )  (e^{s+t}-1 )^2st  (s+t)}\Big(e^s(e^t-1 )^2 (-1-3e^s+e^{s+t}-e^{2s+t}) s^2\\
& + (e^s-1)^2(1-e^t+3e^{s+t}+e^{s+2 t})t^2 -4e^s(e^s-1)(e^t-1)(e^{s+t}-1 )(s-t) \\
&  + (7 e^{s+t}-7 e^{2 (s+t)}-e^{3 (s+t)}+2 e^{3 s+t}+3 e^{3 s+2 t}+e^{2 s+3 t}-3 e^s-2 e^{2 s}-e^t+1)s t \Big),
\end{aligned}
\end{equation*}
$$K_3(s)=\frac{2-2e^{s}+se^s+s}{4s(e^s-1)^2} 
 \qquad H_4(s,t)=\frac{(e^s-1)(e^t-1)(s+t)}{8e^{\frac{s+t}2}(e^{s+t}-1)st}.$$
 The power of $k$ in the sum denoted by $c(i,j)$, counts how many of indices $i,\, j$ are equal to 3. \qed
\end{proposition}

Unlike the phenomena observed for the scalar curvature  in Remark \ref{resultsfromtwodimconformal}, the functions of the heat trace densities of the Laplacian on 1-forms are not related, at least in  the same way as before, to those of the Laplacian on 1-forms of the conformally flat metric. 
This is a consequence of the simple fact that the Laplacian on 1-forms of the product Riemannian manifolds is not the  sum of the Laplacians on 1-forms of the components.   
In fact, if $(M_1, g_1)$ and $(M_2, g_2)$ are two oriented Riemannian manifolds,   
then the Laplacian on 1-forms on the product manifold $(M_1 \times M_2,\, g_1\times g_2)$ is given by
$$\lap_1\otimes 1 + 1\otimes \lap_1 + \lap_0\otimes 1  + 1 \otimes \lap_0 + 2d_0\otimes d_0^{\ast} + 2d_0^{\ast}\otimes d_0 ,$$
where $\lap_0$ and $\lap_1$ are the Laplacians on functions and 1-forms for the corresponding manifolds.

Using the above proposition and Theorem \ref{2nd_scalar_thm}, we obtain the Ricci density in the following 
theorem. 

\begin{theorem}
The Ricci density $\ricden$ of $\nctorus$ equipped with the metric \eqref{metric_2} is given by
\begin{equation*}
\begin{aligned}
\pi^{\frac32}\ricden 
=&-\left(\tilde{K}_{22}^{\nabla}(\delta_2^2(h))+2\tilde{W}_{22}^{\nabla}(\delta_2(h)^2) + k^{-2} \tilde{K}_3^{\nabla}(\delta_3^2(h))+k^{-2}\tilde{H}_3^{\nabla}(\delta_3(h)^2)\right) E_{11}\\
&-\left(\tilde{K}_{11}^{\nabla}(\delta_1^2(h))+2\tilde{W}_{11}^{\nabla}(\delta_1(h)^2)+ k^{-2} \tilde{K}_3^{\nabla}(\delta_3^2(h))+k^{-2}\tilde{H}_3^{\nabla}(\delta_3(h)^2)\right)  E_{22}\\
&- k^{-2}\tilde{H}_4^{\nabla}(\delta_3(h)^2)  E_{33}\\
&-\sum k^{-c(i,j)}\left(\tilde{K}_{ij}^{\nabla}(\delta_i\delta_j(h))+S_{ij}^{\nabla}([\delta_i(h),\delta_j(h)])+\tilde{W}_{ij}^{\nabla}(\{\delta_i(h),\delta_j(h)\})\right)  E_{ij}.
\end{aligned}
\end{equation*}
where $\tilde{K}_{ij}$ (resp. $\tilde{W}_{ij}$ and $\tilde{H}_{ij}$) is different from  ${K}_{ij}$ (resp. $\tilde{W}_{ij}$ and ${H}_{ij}$) only in their diagonal entries. 
The new functions are given by 
$$\tilde{K}_{22}=\tilde{K}_{11}=K_{11}-K_1=\frac{-1+e^s+se^{s/2}}{4s(1+e^{s/2})^2},\qquad \tilde{K}_{33}=K_{33}-K_2=\frac{1}{4e^{s/2}},$$
$$\tilde{K}_{3}=K_{3}-K_2=\frac{-1+e^s+se^{s/2}}{4se^{s/2}(1+e^{s/2})^2},\qquad \tilde{W}_{11}=\tilde{W}_{22}=W_{11}-\frac12H_1,$$
and $\tilde{H}_3=H_3-H_2$,  $\tilde{H}_4=H_4-H_2$, and $\tilde{W}_{33}=W_{33}$.
\qed
\end{theorem}

\begin{figure}[h]
\centering
\includegraphics[scale=0.52]{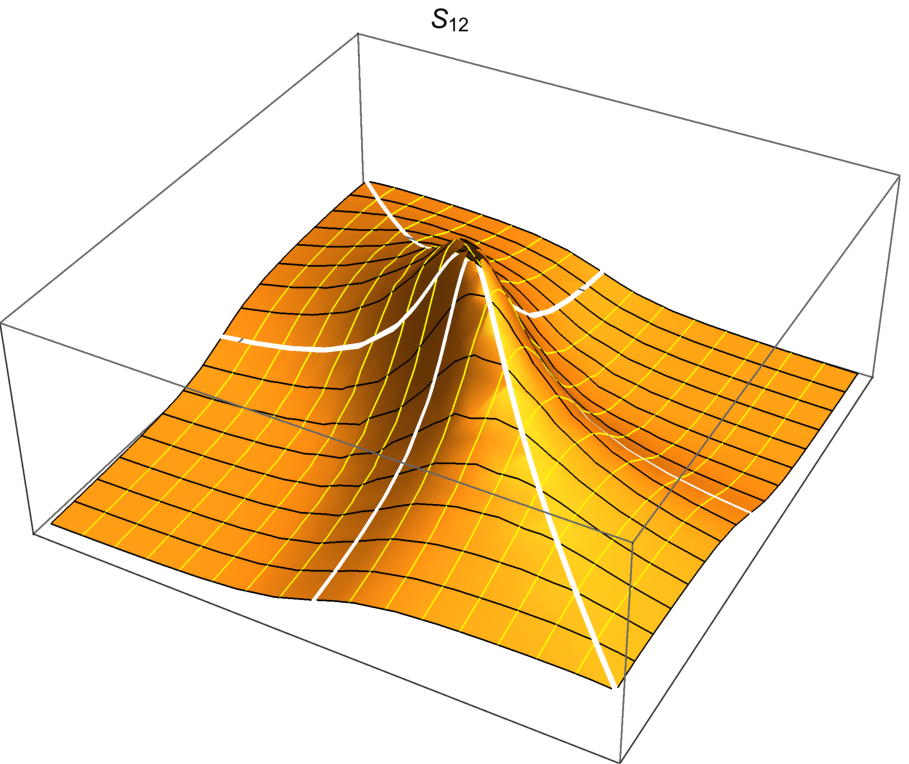}
\includegraphics[scale=0.52]{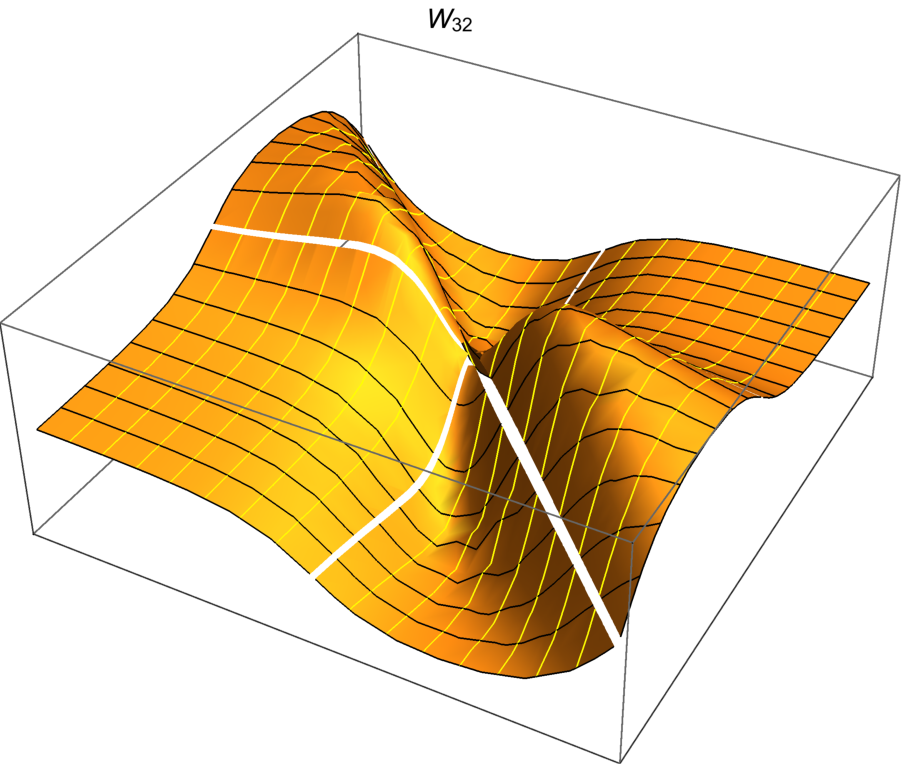}
\includegraphics[scale=0.52]{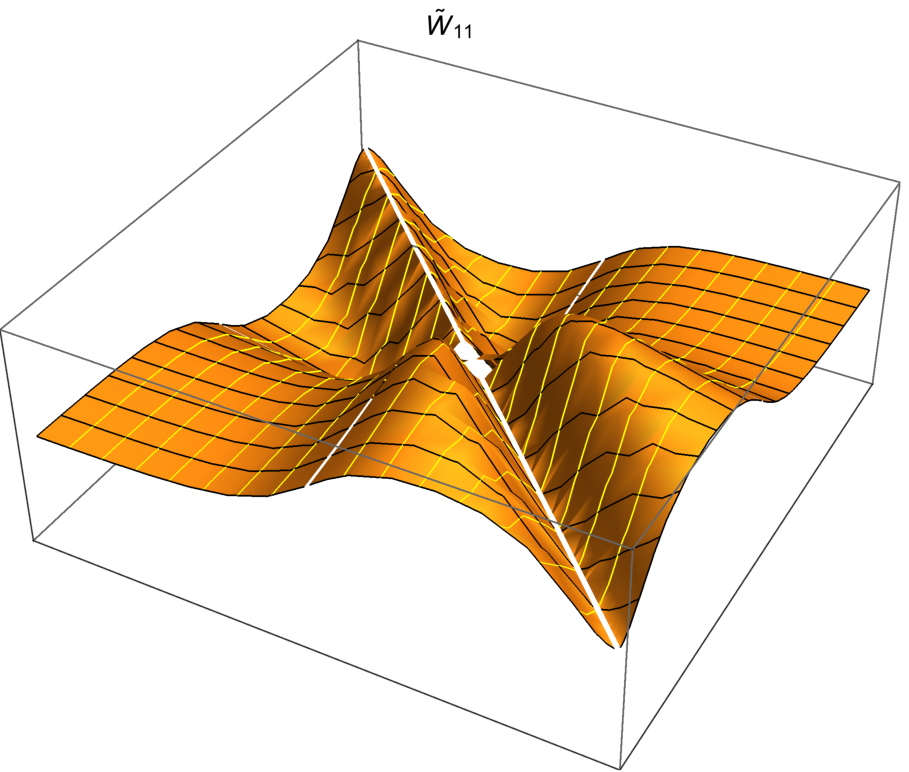}
\caption{The graph of functions $S_{12}$, $W_{32}$ and $\tilde{W}_{11}$.}
\end{figure}

The classical limit of the Ricci density can be obtained by letting  $s, t\to 0.$
First note that  in the commutative case the terms involving functions $S_{ij}$ disappear because they act on the commutator $[\delta_i(h),\delta_j(h)]$ which is zero.
On the other hand, functions $W_{ij}$ are anti-symmetric in their variables; $W_{ij}(s,t)=-W_{ij}(t,s)$.
Hence, the terms involving them will vanish too.
Moreover, since $\lim\limits_{(s,t)\to(0,0)} H_1(s,t)=0$, we have $\lim\limits_{(s,t)\to(0,0)} \tilde{W}_{ij}(s,t)=0.$
The limit of the other terms are given by
\begin{equation*}
\lim_{s\rightarrow 0}{\bf \tilde{K}}(s)=\begin{pmatrix}
\frac18 & 0 & \frac18\\
0 & \frac18 & \frac18\\
\frac18 &\frac18 & \frac14
\end{pmatrix},\qquad \lim_{s\rightarrow 0}\tilde{K}_3(s)=\frac18,
\end{equation*}  
and also
$$\lim_{(s,t)\rightarrow (0,0)}\tilde{H}_{3}(s,t)=-\frac14,\qquad \lim_{(s,t)\rightarrow (0,0)}\tilde{H}_{4}(s,t)=-\frac14.$$
Thus when $\theta\to 0$, the Ricci density $\ricden$ approaches to
\begin{equation*}
\begin{aligned}
&\ricden_{0} =\frac1{8\pi^{\frac32}} \times\\
&\begin{pmatrix}
e^{-2h}(2 \delta _3(h)^2-\delta _3^2(h)) -\sum\limits_{i=1}^2\delta _i^2(h) & 0 & -e^{-h}\delta _1\delta _3(h) \\
 0 & \hspace*{-1cm} e^{-2h}(2 \delta _3(h)^2-\delta _3^2(h)) -\sum\limits_{i=1}^2\delta _i^2(h) &   -e^{-h}\delta _2\delta _3(h)\\
-e^{-h}\delta _1\delta _3(h) & -e^{-h}\delta _2\delta _3(h) & \hspace*{-0.4cm} 2e^{-2h}(\delta _3(h)^2-\delta _3^2(h))
\end{pmatrix},
\end{aligned}
\end{equation*}
while the Ricci density in the classical case is given by
\begin{equation*}
\ricden_{com} =
\begin{pmatrix}
 e^{2h}(h_{11}+h_{22})+h_{33}-2 (h_3)^2& 0 & h_{13} \\
 0 & \hspace*{-0.5cm} e^{2h}(h_{11}+h_{22})+h_{33}-2 (h_3)^2 &  h_{23}\\
e^{2h} h_{13} & e^{2h} h_{23} & \hspace*{-0.3cm} 2h_{33}-2(h_3)^2\\
\end{pmatrix}.
\end{equation*}
The apparent discrepancy between the limit case $\ricden_{0}$ and the commutative formula
$\ricden_{com}$ is due to our convention for the  Ricci functional, and as mentioned in Remark \ref{thediffofcoeff} 
we have the relation 
$$(R_{1,k}J)\ricden_{0}(R_{1,k}J)^{\ast}=R_{1,k}J\ricden_{0}J R_{1,k^{-1}} =\frac1{8\pi^{\frac32}}\ricden_{com}e^{-2h}.$$

\begin{appendices}
\section{Computations}
In this section we give some details of the computation of  the scalar curvature for the non-conformal metric. 
The full details  can be found  in the Mathematica file accompanying this paper.

The computation starts from the formula for $b_2$ given by \eqref{bjs}.
We first plug in formula $b_1$ and write $b_2(\xi,\lambda)$  in terms of $b_1$ and the homogeneous parts of the symbol $a_2$, $a_1$ and $a_0$:
\begin{equation*}
\begin{aligned}
b_2(\xi,\lambda)=&-b_0 a_0b_0
-b_1 a_1b_0
-\partial_1 (b_0)\delta _1 (a_1 )b_0
-\partial_2 (b_0)\delta _2 (a_1)b_0
-\partial_3 b_0\delta _3(a_1)b_0\\
&-\partial_1 (b_1)\delta _1 (a_2 )b_0
-\partial_2 (b_1) \delta _2(a_2)b_0
-\partial_3 (b_1) \delta _3(a_2)b_0\\
&-\frac{1}{2} \partial_1^2 (b_0) \delta^2 (a_2 )b_0
-\frac{1}{2} \partial_2^2(b_0)\delta _2^2(a_2)b_0
-\frac{1}{2} \partial_3^2(b_0) \delta _3^2(a_2)b_0\\
&-\partial_2 \partial_3 (b_0) \delta _3 \delta _2 (a_2 ) b_0
-\partial_1\partial_2 (b_0) \delta _2 \delta _1 (a_2 )b_0
-\partial_1\partial_3 (b_0) \delta _1 \delta _3(a_2 )b_0.
\end{aligned}
\end{equation*}
The next step is to plug  $a_j$'s from Lemma \ref{symboloflap0noncon} into the above formula.
Note that the derivatives of $b_0$ can be written  as  
$$\partial_1 (b_0)=-2\xi_1 k^2 b_0^2,\quad \partial_2 (b_0)=-2\xi_2 k^2 b_0^2,\quad \partial_3 (b_0)=-2\xi_3 b_0^2.$$
The complete outcome is long and involves 465 terms.  Here we only display the result for a  sample term $\partial_3(b_0)\delta_3(a_2)b_0$ below.
\begin{align*}
&\partial_3(b_0)\delta_3(a_2)b_0\\
&=-4 \xi _1^5\xi _3  k^2 b_0^2 \delta _1 (k^2 )b_0^2 \delta _3 (k^2 ) b_0 
-8 \xi _1^5\xi _3  k^2 b_0^3 \delta _1 (k^2 ) b_0 \delta _3 (k^2 ) b_0 
-4  \xi _1^4\xi _3^2  b_0^2 \delta _3 (k^2 ) b_0^2 \delta _3 (k^2 ) b_0\\
&
-8  \xi _1^4\xi _3^2  b_0^3 \delta _3 (k^2 ) b_0 \delta _3 (k^2 ) b_0
+2 \xi _1^4 b_0^2 \delta _3 (k^2 ) b_0 \delta _3 (k^2 ) b_0 
-4  \xi _1^4\xi _2 \xi _3  k^2 b_0^2 \delta _2 (k^2 ) b_0^2 \delta _3 (k^2 ) b_0\\
&
-8 \xi _1^4\xi _2 \xi _3  k^2 b_0^3 \delta _2 (k^2 ) b_0 \delta _3 (k^2 ) b_0 
-8 \xi _1^3\xi _2^2 \xi _3 k^2 b_0^2 \delta _1 (k^2 ) b_0^2 \delta _3 (k^2 ) b_0 
+4 \xi _1^3\xi _3 b_0 k \delta _1(k) b_0^2 \delta _3 (k^2 ) b_0 \\
&
-16 \xi _1^3\xi _2^2 \xi _3 k^2 b_0^3 \delta _1 (k^2 ) b_0 \delta _3 (k^2 ) b_0 
+4  \xi _1^3\xi _3 b_0^2 k \delta _1(k) b_0 \delta _3 (k^2 ) b_0
+4  \xi _1^2\xi _2^2 b_0^2 \delta _3 (k^2 ) b_0 \delta _3 (k^2 ) b_0\\
&
-8 \xi _1^2\xi _2^2 \xi _3^2 b_0^2 \delta _3 (k^2 ) b_0^2 \delta _3 (k^2 ) b_0 
-16 \xi _1^2\xi _2^2 \xi _3^2 b_0^3 \delta _3 (k^2 ) b_0 \delta _3 (k^2 ) b_0 
+2  \xi _1^2\xi _3^2 b_0  k^{-1}  \delta _3(k) b_0^2 \delta _3 (k^2 ) b_0\\
&
-2 \xi _1^2\xi _3^2 b_0 \delta _3(k)  k^{-1}  b_0^2 \delta _3 (k^2 ) b_0 
+2 \xi _1^2\xi _3^2 b_0^2  k^{-1}  \delta _3(k) b_0 \delta _3 (k^2 ) b_0 
-2 \xi _1^2\xi _3^2 b_0^2 \delta _3(k)  k^{-1}  b_0 \delta _3 (k^2 ) b_0 \\
&
- \xi _1^2b_0  k^{-1}  \delta _3(k) b_0 \delta _3 (k^2 ) b_0 
+\xi _1^2b_0 \delta _3(k)  k^{-1}  b_0 \delta _3 (k^2 ) b_0 
-8 \xi _1^2\xi _2^3 \xi _3 k^2 b_0^2 \delta _2 (k^2 ) b_0^2 \delta _3 (k^2 ) b_0 \\
&
-16   \xi _1^2\xi _2^3 \xi _3k^2 b_0^3 \delta _2 (k^2 ) b_0 \delta _3 (k^2 ) b_0
+4 \xi _1^2\xi _2 \xi _3 b_0 k \delta _2(k) b_0^2 \delta _3 (k^2 ) b_0 
+4  \xi _1^2\xi _2 \xi _3 b_0^2 k \delta _2(k) b_0 \delta _3 (k^2 ) b_0\\
&
-4 \xi _1\xi _2^4 \xi _3 k^2 b_0^2 \delta _1 (k^2 ) b_0^2 \delta _3 (k^2 ) b_0 
-8 \xi _1\xi _2^4 \xi _3 k^2 b_0^3 \delta _1 (k^2 ) b_0 \delta _3 (k^2 ) b_0 
+4 \xi _1\xi _2^2 \xi _3 b_0 k \delta _1(k) b_0^2 \delta _3 (k^2 ) b_0 \\
&
+4  \xi _1\xi _2^2 \xi _3b_0^2 k \delta _1(k) b_0 \delta _3 (k^2 ) b_0 
+2 \xi _2^4b_0^2 \delta _3 (k^2 ) b_0 \delta _3 (k^2 ) b_0 
-  \xi _2^2b_0  k^{-1}  \delta _3(k) b_0 \delta _3 (k^2 ) b_0\\
&
+\xi _2^2b_0 \delta _3(k)  k^{-1}  b_0 \delta _3 (k^2 ) b_0 
-4 \xi _2^4 \xi _3^2b_0^2 \delta _3 (k^2 ) b_0^2 \delta _3 (k^2 ) b_0 
-8  \xi _2^4 \xi _3^2b_0^3 \delta _3 (k^2 ) b_0 \delta _3 (k^2 ) b_0\\
&
+2  \xi _2^2 \xi _3^2b_0  k^{-1}  \delta _3(k) b_0^2 \delta _3 (k^2 ) b_0
-2 \xi _2^2 \xi _3^2b_0 \delta _3(k)  k^{-1}  b_0^2 \delta _3 (k^2 ) b_0 
+2\xi _2^2 \xi _3^2 b_0^2  k^{-1}  \delta _3(k) b_0 \delta _3 (k^2 ) b_0 \\
&
-2 \xi _2^2 \xi _3^2b_0^2 \delta _3(k)  k^{-1}  b_0 \delta _3 (k^2 ) b_0 
-4 \xi _2^5 \xi _3k^2 b_0^2 \delta _2 (k^2 ) b_0^2 \delta _3 (k^2 ) b_0 
-8 \xi _2^5 \xi _3k^2 b_0^3 \delta _2 (k^2 ) b_0 \delta _3 (k^2 ) b_0 \\
&
+4 \xi _2^3 \xi _3b_0 k \delta _2(k) b_0^2 \delta _3 (k^2 ) b_0 
+4 \xi _2^3 \xi _3 b_0^2 k \delta _2(k) b_0 \delta _3 (k^2 ) b_0.
\end{align*}

Then  we apply the substitution given in \eqref{substitution} and integrate with respect to $\eta$ and $\theta$.
The result then is 
\begin{align*}
&\frac1{\pi^2}\int_{-\infty}^{+\infty}\int_0^{2\pi} b_2(u,\eta,\theta,-1)\frac{1+\eta^2}{2}d \theta d\eta\\
&=2 u^3k^2 b_0^2 \delta _1(k) k^3 b_0^2 k\delta _1(k) b_0 
+2 u^3k^2 b_0^2 \delta _1(k) k^3 b_0^2 \delta _1(k) k b_0 
+2 u^3k^2 b_0^2\delta _2(k) k^3 b_0^2 k \delta _2(k)b_0 \\
&
+2 u^3 k^2 b_0^2 \delta _2(k) k^3 b_0^2 \delta _2(k) k b_0 
+4 u^3 k^4 b_0^3 k \delta _1(k) b_0 k \delta _1(k) b_0 
+4 u^3 k^4 b_0^3 k \delta _1(k) b_0 \delta _1(k) k b_0 \\
&
+4 u^3k^4 b_0^3 k \delta _2(k) b_0 k \delta _2(k) b_0 
+4u^3k^4 b_0^3 k \delta _2(k) b_0 \delta _2(k) k b_0 
+4u^3k^4 b_0^3 \delta _1(k) k b_0 k \delta _1(k) b_0 \\
&
+4 u^3k^4 b_0^3 \delta _1(k) k b_0 \delta _1(k) k b_0 
+4 u^3k^4 b_0^3 \delta _2(k) k b_0 k \delta _2(k) b_0 
+4 u^3k^4 b_0^3 \delta _2(k) k b_0 \delta _2(k) k b_0 \\
&
+2 u^3k^2b_0^2 k \delta _1(k) k^2 b_0^2 k \delta _1(k) b_0 
+2 u^3k^2 b_0^2k\delta _1(k)k^2 b_0^2 \delta _1(k) k b_0 
-2 u^2k^4 b_0^3 k \delta _1\left(\delta _1(k)\right) b_0 \\
&
+2 u^3k^2 b_0^2k\delta_2(k) k^2 b_0^2 \delta _2(k) k b_0 
+2 u^3k^2 b_0^2 k\delta _2(k)k^2 b_0^2k \delta _2(k) b_0 
-2 u^2 k^4 b_0^3 k \delta _2\left(\delta _2(k)\right) b_0 \\
&
-4u^2k^4 b_0^3 \delta _1(k) \delta _1(k) b_0 
-2 u^2k^4 b_0^3 \delta _1\left(\delta _1(k)\right) k b_0
-4u^2k^4 b_0^3 \delta _2(k) \delta _2(k) b_0 \\
&
-2u^2k^4 b_0^3 \delta_2\left(\delta _2(k)\right) k b_0 
-2u^2b_0^2 k \delta _3(k) b_0 k \delta _3(k) b_0 
-2u^2b_0^2 k \delta _3(k) b_0 \delta _3(k) k b_0 \\
&
+2u^2b_0^2 k \delta _3(k) b_0^2 k \delta _3(k) b_0 
+2u^2b_0^2 k \delta _3(k) b_0^2 \delta _3(k) k b_0 
-2u^2b_0^2 \delta _3(k) k b_0 k \delta _3(k) b_0 \\
&
-2 u^2b_0^2 \delta _3(k) k b_0 \delta _3(k) k b_0
+2u^2b_0^2 \delta _3(k) k b_0^2 k \delta _3(k) b_0 
+2 u^2b_0^2 \delta _3(k) k b_0^2 \delta _3(k) k b_0 \\
&
+4 u^2 b_0^3 k \delta _3(k) b_0 k \delta _3(k) b_0
+4 u^2b_0^3 k \delta _3(k) b_0 \delta _3(k) k b_0
+4 u^2b_0^3 \delta _3(k) k b_0 k \delta _3(k) b_0\\
&
+4 u^2b_0^3 \delta _3(k) k b_0 \delta _3(k) k b_0
-8 u^2k^2 b_0^2 k \delta _1(k) b_0 k \delta _1(k) b_0
-6 u^2k^2 b_0^2 k \delta _1(k) b_0 \delta _1(k) k b_0 \\
&
-8u^2k^2 b_0^2 k \delta _2(k) b_0 k \delta _2(k) b_0 
-6u^2k^2 b_0^2 k \delta _2(k) b_0 \delta _2(k) k b_0
-6 u^2k^2 b_0^2 \delta _1(k) k b_0 k \delta _1(k) b_0\\
&
-4 u^2 k^2 b_0^2 \delta _1(k) k b_0 \delta _1(k) k b_0
-6 u^2 k^2 b_0^2 \delta _2(k) k b_0 k \delta _2(k) b_0
-4 u^2 k^2 b_0^2 \delta _2(k) k b_0 \delta _2(k) k b_0\\
&
-2  u^2b_0 k \delta _1(k) k^2 b_0^2 k \delta _1(k) b_0 
-2 u^2  b_0 k \delta _1(k) k^2 b_0^2 \delta _1(k) k b_0 
-2 u^2  b_0 k \delta _2(k) k^2 b_0^2 k \delta _2(k) b_0 \\
&
-2  u^2 b_0 k \delta _2(k) k^2 b_0^2 \delta _2(k) k b_0 
+ u b_0^2 k \delta _3\left(\delta _3(k)\right) b_0 
+2 u  b_0^2 \delta _3(k) \delta _3(k) b_0 \\
&
+ u b_0^2 \delta _3\left(\delta _3(k)\right) k b_0 
-2  u  b_0^3 k \delta _3\left(\delta _3(k)\right) b_0
-4 u  b_0^3 \delta _3(k) \delta _3(k) b_0 \\
&
-2 u b_0^3 \delta _3\left(\delta _3(k)\right) k b_0 
+3 u k^2 b_0^2 k \delta _1\left(\delta _1(k)\right) b_0
+3 u k^2 b_0^2 k \delta _2\left(\delta _2(k)\right) b_0\\
&
+4 u  k^2 b_0^2 \delta _1(k) \delta _1(k) b_0 
+ u k^2 b_0^2 \delta _1\left(\delta _1(k)\right) k b_0 
+4 u k^2 b_0^2 \delta _2(k) \delta _2(k) b_0\\
&
+  u k^2 b_0^2 \delta _2\left(\delta _2(k)\right) k b_0
+ u b_0  k^{-1} \delta _3(k) b_0 k \delta _3(k) b_0 
+ u b_0  k^{-1} \delta _3(k) b_0 \delta _3(k) k b_0 \\
&
- u b_0  k^{-1} \delta _3(k) b_0^2 k \delta _3(k) b_0 
- u b_0  k^{-1} \delta _3(k) b_0^2 \delta _3(k) k b_0 
+4 u  b_0 k \delta _1(k) b_0 k \delta _1(k) b_0 \\
&
+2  u b_0 k \delta _1(k) b_0 \delta _1(k) k b_0 
+4  u b_0 k \delta _2(k) b_0 k \delta _2(k) b_0 
+2  u b_0 k \delta _2(k) b_0 \delta _2(k) k b_0 \\
&
-  u b_0 \delta _3(k)  k^{-1} b_0 k \delta _3(k) b_0 
- u b_0 \delta _3(k)  k^{-1} b_0 \delta _3(k) k b_0 
+ u b_0 \delta _3(k)  k^{-1} b_0^2 k \delta _3(k) b_0\\
&
+ u b_0 \delta _3(k)  k^{-1} b_0^2 \delta _3(k) k b_0 
- u b_0^2  k^{-1} \delta _3(k) b_0 k \delta _3(k) b_0 
- u b_0^2  k^{-1} \delta _3(k) b_0 \delta _3(k) k b_0 \\
&
- u b_0^2 k \delta _3(k) b_0  k^{-1} \delta _3(k) b_0 
+ u b_0^2 k \delta _3(k) b_0 \delta _3(k)  k^{-1} b_0 
+ u  b_0^2 \delta _3(k)  k^{-1} b_0 k \delta _3(k) b_0\\
&
+ u b_0^2 \delta _3(k)  k^{-1} b_0 \delta _3(k) k b_0 
- u b_0^2 \delta _3(k) k b_0  k^{-1} \delta _3(k) b_0 
+ u b_0^2 \delta _3(k) k b_0 \delta _3(k)  k^{-1} b_0 \\
&
-  b_0  k^{-1} \delta _3\left(\delta _3(k)\right) b_0
-  b_0 k \delta _1\left(\delta _1(k)\right) b_0
-  b_0 k \delta _2\left(\delta _2(k)\right) b_0
+  b_0^2  k^{-1} \delta _3\left(\delta _3(k)\right) b_0\\
&
-  b_0^2 \delta _3\left(\delta _3(k)\right)  k^{-1} b_0
+  b_0 \delta _3(k) k^{-2} \delta _3(k) b_0
+  b_0  k^{-1} \delta _3(k)  k^{-1} \delta _3(k) b_0\\
&
-  b_0^2  k^{-1} \delta _3(k)  k^{-1} \delta _3(k) b_0
+  b_0^2 \delta _3(k)  k^{-1} \delta _3(k)  k^{-1} b_0
+\frac12 b_0k^{-1}\delta_3(k) b_0 k^{-1} \delta_3(k) b_0\\
&
-\frac12 b_0  k^{-1}\delta_3(k)b_0 \delta_3(k)k^{-1} b_0
-\frac12 b_0\delta_3(k)k^{-1} b_0k^{-1} \delta_3(k) b_0
+\frac12 b_0 \delta_3(k)k^{-1} b_0 \delta_3(k)k^{-1} b_0.
\end{align*}

To perform the $u$ integration, we apply  Corollary \ref{rearrangenoncon} where the functions $F^{[v]}_{m_0,\cdots,m_p}$ show up in the result.
The $\rho$ terms appearing in the outcome expression include $\delta_j(k)$ and $\delta^2_j(k)$ multiplied by a power of $k$. 
We use the following identities to bring all these $\rho$'s into the form  $k^{-1}\delta_j$ or $k^{-1}\delta_j^2(k)$.
$$F(\Delta)(\rho_1\rho_2)=F(\Delta_{(1)}\Delta_{(2)})(\rho_1\cdot \rho_2),\qquad F(\Delta)(k^m\rho k^n)=k^{m+n}\Delta^{\frac{n}2} F(\Delta)(\rho),$$
$$F(\Delta_{(1)},\Delta_{(2)})(k^l\rho_1\cdot k^m \rho_2 k^n)=k^{l+m+n}\Delta^{\frac{m+n}2}_{(1)}\Delta^{\frac{n}2}_{(2)}F(\Delta_{(1)}\Delta_{(2)})(\rho_1\cdot \rho_2).$$
These identities are consequences of the fact that $\Delta$ is a $C^\ast$-algebra automorphism which commutes with $k$ and also $xk=k\Delta^{\frac12}(x)$.
Applying the aforementioned identities, the integral of $b_2$, up to the total factor $\pi^2$, is equal to     
\begin{align*}
&\Big((3  + \Delta^{\frac12}) F^{[2]}_{2,1} (\Delta ) \left(k^{-1}\delta_1^2(k) \right)
-  F^{[2]}_{1,1} (\Delta) \left(k^{-1}\delta_1^2(k) \right)
-2 (1+ \Delta^{\frac12}) F^{[2]}_{3,1} (\Delta) \left(k^{-1}\delta_1^2(k) \right)\Big)\\
&+\Big(2   \Delta _{(1)}( \Delta _{(2)}^{\frac12} +2) F^{[2]}_{1,1,1} (\Delta _{(1)},\Delta _{(2)} ) \left(k^{-1}\delta _1(k)\cdot k^{-1}\delta _1(k) \right)\\
&\quad\,\, -2  \Delta _{(1)}^2( \Delta _{(2)}^{\frac12} +1 ) F^{[2]}_{1,2,1} (\Delta _{(1)},\Delta _{(2)} ) \left(k^{-1}\delta _1(k)\cdot k^{-1}\delta _1(k) \right)\\
&\quad\,\, -2\Delta _{(1)} (3    \Delta _{(2)}^{\frac12}+2   \Delta _{(1)}^{\frac12} \Delta _{(2)}^{\frac12}  +3   \Delta _{(1)}^{\frac12} +4 ) 
F^{[2]}_{2,1,1} \left(\Delta _{(1)},\Delta _{(2)} ) (k^{-1}\delta _1(k)\cdot k^{-1}\delta _1(k) \right)\\
&\quad\,\, +2 \Delta _{(1)}^2  ( 1+ \Delta _{(1)}^{\frac12}) ( 1+ \Delta _{(2)}^{\frac12}) 
F^{[2]}_{2,2,1} (\Delta _{(1)},\Delta _{(2)} ) \left(k^{-1}\delta _1(k)\cdot k^{-1}\delta _1(k) \right)\\
&\quad\,\, +4 \Delta _{(1)}   ( 1+ \Delta _{(1)}^{\frac12}) ( 1+ \Delta _{(2)}^{\frac12})  
F^{[2]}_{3,1,1} (\Delta _{(1)},\Delta _{(2)} ) \left(k^{-1}\delta _1(k)\cdot k^{-1}\delta _1(k) \right)\\
&\quad\,\, +4   \Delta _{(1)}^{\frac12} F^{[2]}_{2,0,1} \left(\Delta _{(1)},\Delta _{(2)} ) (k^{-1}\delta _1(k)\cdot k^{-1}\delta _1(k) \right)\\
&\quad\,\, -4   \Delta _{(1)}^{\frac12} F^{[2]}_{3,0,1} \left(\Delta _{(1)},\Delta _{(2)} ) (k^{-1}\delta _1(k)\cdot k^{-1}\delta _1(k) \right)\Big)\\
&+\Big((3 + \Delta^{\frac12}) F^{[2]}_{2,1} (\Delta) \left(k^{-1}\delta_2^2(k) \right)
-  F^{[2]}_{1,1} (\Delta ) \left(k^{-1}\delta_2^2(k) \right)
-2 (1+   \Delta^{\frac12}) F^{[2]}_{3,1} (\Delta) \left(k^{-1}\delta_2^2(k) \right)\Big)\\
&+\Big(2   \Delta _{(1)}( \Delta _{(2)}^{\frac12} +2) F^{[2]}_{1,1,1} (\Delta _{(1)},\Delta _{(2)} ) \left(k^{-1}\delta _2(k)\cdot k^{-1}\delta _2(k) \right)\\
&\quad\,\, -2  \Delta _{(1)}^2( \Delta _{(2)}^{\frac12} +1 ) F^{[2]}_{1,2,1} (\Delta _{(1)},\Delta _{(2)} ) \left(k^{-1}\delta _2(k)\cdot k^{-1}\delta _2(k) \right)\\
&\quad\,\, -2\Delta _{(1)} (3    \Delta _{(2)}^{\frac12}+2   \Delta _{(1)}^{\frac12} \Delta _{(2)}^{\frac12}  +3   \Delta _{(1)}^{\frac12} +4 ) 
F^{[2]}_{2,1,1} \left(\Delta _{(1)},\Delta _{(2)} ) (k^{-1}\delta _2(k)\cdot k^{-1}\delta _2(k) \right)\\
&\quad\,\, +2 \Delta _{(1)}^2 ( 1+ \Delta _{(1)}^{\frac12}) ( 1+ \Delta _{(2)}^{\frac12}) 
F^{[2]}_{2,2,1} (\Delta _{(1)},\Delta _{(2)} ) \left(k^{-1}\delta _2(k)\cdot k^{-1}\delta _2(k) \right)\\
&\quad\,\, +4 \Delta _{(1)}  ( 1+ \Delta _{(1)}^{\frac12}) ( 1+ \Delta _{(2)}^{\frac12}) 
F^{[2]}_{3,1,1} (\Delta _{(1)},\Delta _{(2)} ) \left(k^{-1}\delta _2(k)\cdot k^{-1}\delta _2(k) \right)\\
&\quad\,\, +4   \Delta _{(1)}^{\frac12} F^{[2]}_{2,0,1} \left(\Delta _{(1)},\Delta _{(2)} ) (k^{-1}\delta _2(k)\cdot k^{-1}\delta _2(k) \right)\\
&\quad\,\, -4   \Delta _{(1)}^{\frac12} F^{[2]}_{3,0,1} \left(\Delta _{(1)},\Delta _{(2)} ) (k^{-1}\delta _2(k)\cdot k^{-1}\delta _2(k) \right)\Big)\\
&+k^{-2}\Big(- F^{[2]}_{1,1} (\Delta) \left(k^{-1}\delta^2 _3(k)\right)
+(1+ \Delta^{\frac12}) F^{[2]}_{2,1} (\Delta) \left(k^{-1}\delta^2_3(k)\right)\\
&\qquad\qquad 
-2(1+ \Delta^{\frac12}) F^{[3]}_{3,1} (\Delta) \left(k^{-1}\delta^2 _3(k)\right)
 +(1-\Delta^{-\frac12})  F^{[3]}_{2,1} (\Delta) \left(k^{-1}\delta _3 ^2(k)\right)\Big)
\\
&+k^{-2}\Big(
(\Delta _{(1)} - \Delta _{(1)}^{\frac12})(\Delta _{(2)}^{\frac12} +1) F^{[2]}_{1,1,1} (\Delta _{(1)},\Delta _{(2)} ) \left(k^{-1}\delta _3(k)\cdot k^{-1}\delta _3(k)\right )\\
&\qquad\quad-2   \Delta _{(1)} (\Delta _{(1)}^{\frac12}+1)( \Delta _{(2)}^{\frac12} +1) F^{[2]}_{2,1,1} (\Delta _{(1)},\Delta _{(2)} ) \left(k^{-1}\delta _3(k)\cdot k^{-1}\delta _3(k)\right )\\
&\qquad\quad+2   \Delta _{(1)}^{\frac12} F^{[2]}_{2,0,1} (\Delta _{(1)},\Delta _{(2)} ) \left(k^{-1}\delta _3(k)\cdot k^{-1}\delta _3(k)\right )\\
&\qquad\quad+(1+\Delta_1^{-\frac12}) F^{[2]}_{1,0,1} (\Delta _{(1)},\Delta _{(2)} ) \left(k^{-1}\delta _3(k)\cdot k^{-1}\delta _3(k)\right )\\
&\qquad\quad+ (\Delta _{(1)}^{\frac12}- \Delta _{(1)}) (\Delta _{(2)}^{\frac12} +1) F^{[3]}_{1,2,1} (\Delta _{(1)},\Delta _{(2)} ) \left(k^{-1}\delta _3(k)\cdot k^{-1}\delta _3(k)\right )\\
&\qquad\quad- 
(\Delta _{(1)}^{\frac12}-\Delta_{(2)}^{-\frac12})(1+\Delta _{(1)}^{\frac12}-\Delta_{(2)}^{\frac12}+\Delta _{(1)}^{\frac12}\Delta_{(2)}^{\frac12}) F^{[3]}_{2,1,1} (\Delta _{(1)},\Delta _{(2)} ) \left(k^{-1}\delta _3(k)\cdot k^{-1}\delta _3(k)\right )\\
&\qquad\quad+2   \Delta _{(1)} (\Delta _{(1)}^{\frac12}+1)( \Delta _{(2)}^{\frac12} +1) F^{[3]}_{2,2,1} (\Delta _{(1)},\Delta _{(2)} ) \left(k^{-1}\delta _3(k)\cdot k^{-1}\delta _3(k)\right )\\
&\qquad\quad+4   \Delta _{(1)} (\Delta _{(1)}^{\frac12}+1)( \Delta _{(2)}^{\frac12} +1) F^{[3]}_{3,1,1} (\Delta _{(1)},\Delta _{(2)} ) \left(k^{-1}\delta _3(k)\cdot k^{-1}\delta _3(k)\right )\\
&\qquad\quad+\frac{1}{2} (\Delta_{(1)}^{-\frac12}-1)(\Delta_2^{-\frac12}-1)  F^{[3]}_{1,1,1} (\Delta _{(1)},\Delta _{(2)} ) \left(k^{-1}\delta _3(k)\cdot k^{-1}\delta _3(k)\right )\\
&\qquad\quad- 4  \Delta _{(1)}^{\frac12} F^{[3]}_{3,0,1} (\Delta _{(1)},\Delta _{(2)} ) \left(k^{-1}\delta _3(k)\cdot k^{-1}\delta _3(k)\right )\\
&\qquad\quad -(1- \Delta_1^{-\frac12}\Delta_2^{-\frac12})  F^{[3]}_{2,0,1} (\Delta _{(1)},\Delta _{(2)} ) \left(k^{-1}\delta _3(k)\cdot k^{-1}\delta _3(k)\right )\Big).
\end{align*}

In the above formula, we grouped the terms with the same sequence of $\rho_j$'s together.
The terms which has  $k^{-1}\delta_1^2(k)$ have exactly the exactly the  same functions  as  the  term $k^{-1}\delta_2^2(k)$,  and it reads
$$\pi^2\Big( (3  + \Delta^{\frac12}) F^{[2]}_{2,1} (\Delta ) -  F^{[2]}_{1,1} (\Delta) -2 (1+ \Delta^{\frac12}) F^{[2]}_{3,1} (\Delta)\Big).$$
If we substitute  the  functions $F^{[v]}_{m_0,m_1}$ in the above expression, we get: 
\begin{equation*}
\psi_1(s_1):=-\frac{\pi ^2 \sqrt{s_1}(s_1\log(s_1)+\log(s_1)-2s_1+2)}{(\sqrt{s_1}-1)^3(\sqrt{s_1}+1)^2}.
\end{equation*}
The function for  $(k^{-1} \delta _1(k)\cdot k^{-1} \delta _1(k))$ is the same as the function for $(k^{-1} \delta _1(k)\cdot k^{-1} \delta _1(k))$ and it is given by  
\begin{align*}
&\phi_1(s_1,s_2)=\frac{2 \pi^2 \sqrt{s_1} \sqrt{s_2}}{ (\sqrt{s_1}-1 ) (s_1-1 )  (\sqrt{s_2}-1 )(s_2-1 )  (\sqrt{s_1 s_2}-1 )  (s_1s_2-1 )^2}\times\\
&\Big(1+s_1^{3/2}  (s_2^{5/2}-\sqrt{s_2}+2 s_2^{3/2} \log  (s_2 )+s_2^2  (\log  (s_1 s_2 )-2 )-2 s_2 \log  (s_1 s_2 )+\log  (s_1 s_2 )+2 )\\
&-s_2+s_2 s_1^{5/2}  (s_2^{3/2}  (\log  (s_1 )-1 )-s_2  (\log  (s_1 s_2 )-2 )+\log  (s_1 s_2 )-\sqrt{s_2}  (\log  (s_1 s_2 )-1 )-2 )\\
&+\log  (s_2 )+s_2 s_1^2  (s_2  (\log  (s_2 )-1 )+1 )-s_1  (s_2^2  (\log  (s_1 s_2 )-1 )-2 s_2 \log  (s_1 )+\log  (s_1 s_2 )+1 )\\
&-\sqrt{s_1}  (s_2^{3/2}  (\log  (s_1 s_2 )+1 )-s_2  (\log  (s_1 s_2 )+2 )+\log  (s_1 s_2 )-\sqrt{s_2}  (\log  (s_1 )+1 )+2 )\Big).
\end{align*}
Also, the functions for $k^{-1} \delta _3(k)k^{-1} \delta _3(k)$ and $k^{-1} \delta^2 _3(k)$  are given by 
$$\psi_2(s_1)=\frac{2 \pi ^2 \left(-s_1^2+2 s_1 \log \left(s_1\right)+1\right)}{ \sqrt{s_1} \left(s_1-1\right)^2  \log \left(s_1\right)},$$
and
\begin{align*}
\phi_2(&s_1,s_2)=\frac{2 \pi ^2}{ \sqrt{s_1s_2}  (s_1-1 )  (s_2-1 )(s_1 s_2-1 )^2 \log  (s_1 ) \log  (s_2 ) \log  (s_1 s_2 )}\\
&\Big( (s_1-1 )^2(  s_2^3 s_1+s_1s_2^2-s_2^2 +3s_2 )  \log  (s_2 ) ^2  -(s_2-1 )^2 ( 3 s_1^2s_2-  s_2 s_1+s_1+1 )  \log  (s_1 )^2   \\
&\quad +    (  s_1^2s_2^2-5s_1s_2^2+ s_1^2s_2+2s_2^2  +4s_1s_2-5s_2+s_1+1 )(s_1 s_2-1 ) \log  (s_1 ) \log  (s_2 )\\
& \quad +  2  (s_2-1 )(s_1-1 )  (s_1^2 s_2^2-1 ) \log  (\frac{s_2}{s_1} ) \Big).
\end{align*}

Finally, we would like to express the result in term of $\log k$ and $\nabla:=\log \Delta=[-2h,\cdot]$. 
To do so we first need to use the formula \eqref{translatetologk}, then replace $\Delta$ with $e^{\nabla}$.
For example, the term involving $\delta_1^2(\log(k))$ comes from  $\psi_1(\Delta)(k^{-1}\delta_1^2(k))$ and it is given by
\begin{align*}
\psi_1(\Delta)f(\Delta)(\delta_1^2(\log(k)))
&= -\frac{\pi ^2 \sqrt{\Delta}(s_1\log(\Delta)+\log(\Delta)-2\Delta+2)}{(\sqrt{\Delta}-1)^3(\sqrt{\Delta}+1)^2}\frac{2(\sqrt{\Delta}-1)}{\log(\Delta)}\left(\delta_1^2(\log(k))\right)\\
&=2\pi^2 \frac{e^{\frac{\nabla}2}(2e^\nabla-\nabla e^\nabla-2-\nabla)}{\nabla(e^\nabla-1)^2}\nabla.
\end{align*}
Multiplying the overall factor $(4\pi)^{-\frac32}$ and factoring out the powers of $\pi$, we get the function $K_1(s)$ given in Theorem \ref{ncconformalscalarcurvature}.
Similarly, other function are obtained as  
\begin{align*}
K_2(s)&=\frac1{8\pi^2}\psi_2(e^{s})f(s)\\
H_1(s,t)&=\frac1{8\pi^2}\left(\phi_1(e^s,e^t) f(e^s) f(e^t) + 2 \psi_1(e^s e^t)g(e^s,e^t)\right),\\
H_2(s,t)&=\frac1{8\pi^2}\left(\phi_2(e^s,e^t) f(e^s) f(e^t) + 2 \psi_2(e^s e^t)g(e^s,e^t)\right).
\end{align*}

\section{Functions from the  rearrangement lemma}\label{appenfunctionsfromrearrange}
In this appendix we list all the functions obtained from the rearrangement lemmas  \ref{rearrangecon} and \ref{rearrangenoncon} which are required in the computations. 
First we have the functions from the conformally flat case in section \ref{conformalmetrics}:
\begin{align*}
F_{1,1}(s_1)&={\pi }/\left({{s_1^{2/3}}+\sqrt[3]{s_1}}\right),\\
F_{2,1}(s_1)&=\pi  \left(\sqrt[3]{s_1}+2\right)/\left(2 \left(\sqrt[3]{s_1}+1\right)^2 \sqrt[3]{s_1}\right),\\
F_{3,1}(s_1)&={\pi  \left(3 s_1^{2/3}+9 \sqrt[3]{s_1}+8\right)}/\left({8 \left(\sqrt[3]{s_1}+1\right)^3 \sqrt[3]{s_1}}\right),\\
F_{1,1,1}(s_1,s_2)&={\pi  \left(\sqrt[3]{s_1} \left(\sqrt[3]{s_2}+1\right)+1\right)}/\Big({\left(\sqrt[3]{s_1}+1\right) s_1 \left(\sqrt[3]{s_2}+1\right) \sqrt[3]{s_2} \left(\sqrt[3]{s_1} \sqrt[3]{s_2}+1\right)}\Big),\\
F_{1,2,1}(s_1,s_2)&=\frac{\pi  \left(2 s_1^{2/3} \left(\sqrt[3]{s_2}+1\right)^2+\sqrt[3]{s_1} \left(\sqrt[3]{s_2}+2\right)^2+\sqrt[3]{s_2}+2\right)}{2 \left(\sqrt[3]{s_1}+1\right)^2 s_1^{5/3} \left(\sqrt[3]{s_2}+1\right)^2 \sqrt[3]{s_2} \left(\sqrt[3]{s_1} \sqrt[3]{s_2}+1\right)},\\
F_{2,1,1}(s_1,s_2)&=\frac{\pi  \left(\left(\sqrt[3]{s_1}+2\right) \sqrt[3]{s_1} \left(\sqrt[3]{s_2}+1\right) \left(\sqrt[3]{s_1} \sqrt[3]{s_2}+2\right)+2\right)}{2 \left(\sqrt[3]{s_1}+1\right)^2 s_1 \left(\sqrt[3]{s_2}+1\right) \sqrt[3]{s_2} \left(\sqrt[3]{s_1} \sqrt[3]{s_2}+1\right)^2},\\
F_{2,2,1}(s_1,s_2)&=\frac{\pi}{2 \left(\sqrt[3]{s_1}+1\right)^3 s_1^{5/3} (\sqrt[3]{s_2}+1)^2 \sqrt[3]{s_2}(\sqrt[3]{s_1} \sqrt[3]{s_2}+1)^2}\Big(\sqrt[3]{s_1} (2 s_2^{2/3}+7 \sqrt[3]{s_2}+6)\\
&\qquad \qquad +(\sqrt[3]{s_2}+1)^2 (s_1^{4/3} \sqrt[3]{s_2}+s_1^{2/3} (\sqrt[3]{s_2}+6)+s_1 (3 \sqrt[3]{s_2}+2))+\sqrt[3]{s_2}+2\Big),\\
F_{3,1,1}(s_1,s_2)&=
\frac{\pi}{8 s_1\sqrt[3]{s_2} (\sqrt[3]{s_1}+1)^3  (\sqrt[3]{s_2}+1) (\sqrt[3]{s_1 s_2}+1)^3 }\Big((9 s_1^{4/3} \sqrt[3]{s_2}+24 s_1^{2/3}) \left(\sqrt[3]{s_2}+1\right){}^2\\
&\qquad +(24 \sqrt[3]{s_1}+3 s_2^{2/3} s_1^{5/3}+27 \sqrt[3]{s_2} s_1+8 s_2^{2/3} s_1+8 s_1) \left(\sqrt[3]{s_2}+1\right)+8\Big)
\end{align*}

The list of functions required in the computations for the non-conformal metric  is the following:
\begin{align*}
F^{[2]}_{1,1}(s_1)&={\log (s_1)}/(s_1-1),\\
F^{[2]}_{2,1}(s_1)&={(s_1-\log (s_1)-1)}/{(s_1-1)^2},\\
F^{[2]}_{3,1}(s_1)&={\left((s_1-4) s_1+2 \log (s_1)+3\right)}/{(2 (s_1-1)^3)},\\
F^{[3]}_{3,1}(s_1)&={(s_1^2-2 s_1 \log (s_1)-1)}/{(2 (s_1-1)^3)},\\
F^{[3]}_{2,1}(s_1)&={(s_1 (\log (s_1)-1)+1)}/{(s_1-1)^2},\\
F^{[2]}_{1,0,1}(s_1,s_2)&={\log \left(s_1 s_2\right)}/{(s_1 s_2-1)},\\
F^{[2]}_{1,1,1}(s_1,s_2)&={\big((s_1 s_2-1) \log (s_1)-(s_1-1) \log (s_1 s_2)\big)}/\big({(s_1-1) s_1 (s_2-1) (s_1 s_2-1)}\big),\\
F^{[2]}_{2,0,1}(s_1,s_2)&=\big({s_1 s_2-\log (s_1 s_2)-1}\big)/{(s_1 s_2-1)^2},\\
F^{[2]}_{1,2,1}(s_1,s_2)&=\frac{1}{(s_1-1)^2 s_1^2 (s_2-1)^2 (s_1 s_2-1)}\Big((s_1-1)^2 \log (s_1 s_2)\\
&\qquad\qquad +(s_1 s_2-1) (s_1 (-s_2)+(s_1 (s_2-2)+1) \log (s_1)+s_1+s_2-1)\Big),\\
F^{[2]}_{2,1,1}(s_1,s_2)&=\frac{1}{(s_1-1)^2 s_1 (s_2-1) (s_1 s_2-1)^2}\Big((s_1-1)^2 \log (s_1 s_2)\\
&\qquad \qquad +(s_1 s_2-1) ((s_1-1) s_1 (s_2-1)+(1-s_1 s_2) \log (s_1))\Big),\\
F^{[2]}_{2,2,1}(s_1,s_2)&=\frac{1}{(s_1-1)^3 s_1^2 (s_2-1)^2 (s_1 s_2-1)^2}\Big(-(s_1 s_2-1) (s_1 (2 s_2-3)+1) \log (s_1))\\
& +(s_1 s_2-1) ((s_1-1) (s_2-1) ( s_1^2 (s_2-1)+s_2s_1-1)-(s_1-1)^3 \log (s_1 s_2)\Big),\\
F^{[2]}_{3,0,1}(s_1,s_2)&=\Big({(s_1 s_2-3) (s_1 s_2-1)+2 \log (s_1 s_2)}\Big)/{(2 (s_1 s_2-1)^3)},\\
F^{[2]}_{3,1,1}(s_1,s_2)&=\frac{1}{2 (s_1-1)^3 s_1 (s_2-1) (s_1 s_2-1)^3}\Big(2 (s_1 s_2-1)^3 \log (s_1)-2 (s_1-1)^3 \log (s_1 s_2)\\
&\qquad\qquad +s_1 (s_1-1) (s_2-1) (s_1 s_2-1) ((s_1-3) s_1 s_2-3 s_1+5)\Big),\\
F^{[3]}_{1,1,1}(s_1,s_2)&=\Big({(-s_1 s_2+1)\log (s_1)+(s_1 -1)s_2 \log (s_1 s_2)}\Big)/\big({(s_1-1) (s_2-1) (s_1 s_2-1)}\big),\\
F^{[3]}_{2,0,1}(s_1,s_2)&=\Big({-s_1 s_2+s_1 s_2 \log (s_1 s_2)+1}\Big)/{(s_1 s_2-1)^2},\\
F^{{[3]}}_{3,0,1}(s_1,s_2)&=\Big({s_1^2 s_2^2-2 s_1 s_2 \log (s_1 s_2)-1}\Big)/({2 (s_1 s_2-1)^3)},\\
F^{[3]}_{1,2,1}(s_1,s_2&)=\frac{(s_1 s_2-1) ((s_1-1) (s_2-1)+(s_1-s_2) \log (s_1))-(s_1-1)^2 s_2 \log (s_1 s_2)}{(s_1-1)^2 s_1 (s_2-1)^2 (s_1 s_2-1)},\\
F^{[3]}_{2,1,1}(s_1,s_2)&=\frac{(s_1 s_2-1)^2 \log (s_1)-(s_1-1) ((s_2-1) (s_1 s_2-1)+(s_1-1) s_2 \log (s_1 s_2))}{(s_1-1)^2 (s_2-1) (s_1 s_2-1)^2},\\
F^{[3]}_{2,2,1}(s_1,s_2)&=\frac{1}{(s_1-1 )^3 s_1  (s_2-1 )^2  (s_1 s_2-1 )^2}
\Big(
s_1^2+s_2^3 s_1^3  (\log(s_1)\hspace{-1mm}-2 )+s_2^2 s_1^3  (3-2 \log  (s_1 ) )\\
&+s_2 s_1^3  (\log  (s_1 s_2 )-1 )+s_2^3 s_1^2  (\log  (s_1 )+2 )-s_2^2 s_1^2  (2 \log  (s_1 ) )\\
&+s_2 s_1^2  (4 \log  (s_1 )-3 \log  (s_1 s_2 )-3 )+s_2^2 s_1  (-2 \log  (s_1 )-3 )-s_1  (2 \log  (s_1 ) )\\
&+s_2 s_1  (\log  (s_1 )+3 \log  (s_1 s_2 )+3 )+s_2  (\log  (s_1 )-\log  (s_1 s_2 )+1 )-1
\Big)\\
F^{[3]}_{3,1,1}(s_1,s_2)&=\frac{1}{2 (s_1-1)^3 (s_2-1) (s_1 s_2-1)^3}\Big(2 (s_1-1)^3 s_2 \log (s_1 s_2)-2 (s_1 s_2-1)^3 \log (s_1)\\
&\qquad +(s_1-1) (s_2-1) (s_1 s_2-1) ((s_1+1) s_1 s_2+s_1-3)\Big).
\end{align*}

\end{appendices}

\addcontentsline{toc}{section}{References}

\def\polhk#1{\setbox0=\hbox{#1}{\ooalign{\hidewidth
  \lower1.5ex\hbox{`}\hidewidth\crcr\unhbox0}}}
\providecommand{\bysame}{\leavevmode\hbox to3em{\hrulefill}\thinspace}
\providecommand{\MR}{\relax\ifhmode\unskip\space\fi MR }
% \MRhref is called by the amsart/book/proc definition of \MR.
\providecommand{\MRhref}[2]{%
  \href{http://www.ams.org/mathscinet-getitem?mr=#1}{#2}
}
\providecommand{\href}[2]{#2}

%\bibliographystyle{amsplain}
%\bibliography{References}

\begin{thebibliography}{10}

\bibitem{Tanvir-Marcolli2012}
Tanvir~Ahamed Bhuyain and Matilde Marcolli, \emph{The {R}icci flow on
  noncommutative two-tori}, Lett. Math. Phys. \textbf{101} (2012), no.~2,
  173--194. \MR{2947960}

\bibitem{Connes1980}
Alain Connes, \emph{{$C^{\ast} $} alg\`ebres et g\'eom\'etrie
  diff\'erentielle}, C. R. Acad. Sci. Paris S\'er. A-B \textbf{290} (1980),
  no.~13, A599--A604. \MR{572645}

\bibitem{Connes1994}
\bysame, \emph{Noncommutative geometry}, Academic Press, Inc., San Diego, CA,
  1994. \MR{1303779}

\bibitem{Connes-Fathizadeh2016}
Alain {Connes} and Farzad {Fathizadeh}, \emph{{The term a\_4 in the heat kernel
  expansion of noncommutative tori}}, arXiv:1611.09815 [math.QA] (2016).

\bibitem{Connes-Moscovici2014}
Alain Connes and Henri Moscovici, \emph{Modular curvature for noncommutative
  two-tori}, J. Amer. Math. Soc. \textbf{27} (2014), no.~3, 639--684.
  \MR{3194491}

\bibitem{Connes-Tretkoff2011}
Alain Connes and Paula Tretkoff, \emph{The {G}auss-{B}onnet theorem for the
  noncommutative two torus}, Noncommutative geometry, arithmetic, and related
  topics, Johns Hopkins Univ. Press, Baltimore, MD, 2011, pp.~141--158.
  \MR{2907006}

\bibitem{Dbrowsli-Sitarz2015}
Ludwik Dabrowski and Andrzej Sitarz, \emph{An asymmetric noncommutative torus},
  SIGMA Symmetry Integrability Geom. Methods Appl. \textbf{11} (2015), Paper
  075, 11. \MR{3402793}

\bibitem{Fathizadeh-Khalkhali2012}
Farzad Fathizadeh and Masoud Khalkhali, \emph{The {G}auss-{B}onnet theorem for
  noncommutative two tori with a general conformal structure}, J. Noncommut.
  Geom. \textbf{6} (2012), no.~3, 457--480. \MR{2956317}

\bibitem{Fathizadeh-Khalkhali2013}
\bysame, \emph{Scalar curvature for the noncommutative two torus}, J.
  Noncommut. Geom. \textbf{7} (2013), no.~4, 1145--1183. \MR{3148618}

\bibitem{Fathizadeh-Khalkhali2015}
\bysame, \emph{Scalar curvature for noncommutative four-tori}, J. Noncommut.
  Geom. \textbf{9} (2015), no.~2, 473--503. \MR{3359018}

\bibitem{Floricel-Ghorbanpour-Khalkhali2016}
Remus {Floricel}, Asghar {Ghorbanpour}, and Masoud {Khalkhali}, \emph{{The
  Ricci Curvature in Noncommutative Geometry}}, Accepted to appear in Journal
  of Nonccommutative Geometry, arXiv:1612.06688 [math.QA] (2016).

\bibitem{Gilkey1995}
Peter~B. Gilkey, \emph{Invariance theory, the heat equation, and the
  {A}tiyah-{S}inger index theorem}, second ed., Studies in Advanced
  Mathematics, CRC Press, Boca Raton, FL, 1995. \MR{1396308}

\bibitem{Khalkhali-Motadelro-Sadeghi2016}
Masoud {Khalkhali}, Ali {Moatadelro}, and Sajad {Sadeghi}, \emph{{A Scalar
  Curvature Formula For the Noncommutative 3-Torus}}, arXiv:1610.04740
  [math.OA] (2016).

\bibitem{Khalhali-Sitarz2018}
Masoud Khalkhali and Andrzej Sitarz, \emph{Gauss-{B}onnet for matrix
  conformally rescaled {D}irac}, J. Math. Phys. \textbf{59} (2018), no.~6,
  063505, 10. \MR{3815127}

\bibitem{Lesch2017}
Matthias Lesch, \emph{Divided differences in noncommutative geometry:
  rearrangement lemma, functional calculus and expansional formula}, J.
  Noncommut. Geom. \textbf{11} (2017), no.~1, 193--223. \MR{3626561}

\bibitem{Lesch-Moscovici2016}
Matthias Lesch and Henri Moscovici, \emph{Modular {C}urvature and {M}orita
  {E}quivalence}, Geom. Funct. Anal. \textbf{26} (2016), no.~3, 818--873.
  \MR{3540454}

\bibitem{Liu2015}
Yang {Liu}, \emph{Modular curvature for toric noncommutative manifolds},
  arXiv:1510.04668 [math.QA] (2015).

\end{thebibliography}
\Addresses
\end{document}